\documentclass[final,onefignum,onetabnum]{siamart220329}

\usepackage{amsfonts}
\usepackage{graphicx}

\newsiamremark{remark}{Remark}
\newsiamremark{hypothesis}{Hypothesis}
\crefname{hypothesis}{Hypothesis}{Hypotheses}
\newsiamthm{claim}{Claim}

\headers{Proximal Method for Sum-of-Ratios Problems}{R. I. Bo\c{t}, M. N. Dao, and G. Li}

\title{Inertial Proximal Block Coordinate Method for a Class of Nonsmooth Sum-of-Ratios Optimization Problems\thanks{Submitted to the editors DATE.
\funding{The first author was partially supported by the Austrian Science Fund (FWF), project number W1260-N35. The second author was partially supported by the Australian Research Council (ARC), project number DP190100555 and by a public grant as part of the Investissement d'avenir project, reference ANR-11-LABX-0056-LMH, LabEx LMH. The third author was partially supported by the Australian Research Council (ARC), project number DP190100555.}}}

\author{
Radu Ioan Bo\c{t}\thanks{Faculty of Mathematics, University of Vienna, A-1090 Vienna, Austria (\texttt{radu.bot@univie.ac.at}).}
\and
Minh N. Dao\thanks{School of Science, RMIT University, Melbourne 3000, Australia (\texttt{minh.dao@rmit.edu.au}).}
\and
Guoyin Li\thanks{Department of Applied Mathematics, University of New South Wales, Sydney 2052, Australia (\texttt{g.li@unsw.edu.au}).}
}

\usepackage{amsopn}

\ifpdf
\hypersetup{
  pdftitle={Inertial Proximal Block Coordinate Method for a Class of Nonsmooth  Sum-of-Ratios Optimization Problems},
  pdfauthor={R. I. Bo\c{t}, M. N. Dao, and G. Li}
}
\fi

\usepackage{amsmath, amssymb}
\usepackage[breakable]{tcolorbox}

\newsiamremark{example}{Example}
\newsiamremark{algo}{Algorithm}
\newsiamremark{assumption}{Assumption}

\renewcommand\theenumi{(\roman{enumi})}
\renewcommand\theenumii{(\alph{enumii})}

\usepackage[shortlabels]{enumitem}

\newcommand{\RP}{\ensuremath{\mathbb{R}_+}}
\newcommand{\RPP}{\ensuremath{\mathbb{R}_{++}}}
\newcommand{\st}{\ensuremath{\stackrel}}

\newcommand{\argmax}{\ensuremath{\operatorname*{argmax}}}
\newcommand{\argmin}{\ensuremath{\operatorname*{argmin}}}

\newcommand{\dom}{\ensuremath{\operatorname{dom}}}

\newcommand{\prox}{\ensuremath{\operatorname{Prox}}}

\newcommand{\dist}{\ensuremath{\operatorname{dist}}}
\newcommand{\supp}{\ensuremath{\operatorname{supp}}}

\allowdisplaybreaks

\newcommand{\scal}[2]{\left\langle {#1},{#2} \right\rangle}

\newcommand{\x}{\mathbf{x}}
\newcommand{\y}{\mathbf{y}}
\newcommand{\z}{\mathbf{z}}
\newcommand{\ba}{\mathbf{a}}
\newcommand{\bu}{\mathbf{u}}
\newcommand{\bv}{\mathbf{v}}
\newcommand{\bw}{\mathbf{w}}
\newcommand{\blambda}{\boldsymbol{\lambda}}
\newcommand{\bmu}{\boldsymbol{\mu}}

\newcounter{step}
\newcommand\step[1]{%
	\refstepcounter{step}	
	\vskip 0.25\baselineskip
	\ifx\hfuzz#1\hfuzz
		\item[~\(\triangleright\)~\textbf{Step~\arabic{step}.}]
	\else
		\item[~\(\triangleright\)~\textbf{Step~\arabic{step}}] (\texttt{#1})\textbf{.}%
	\fi
}

\begin{document}

\maketitle

\begin{abstract}
In this paper, we consider a class of nonsmooth sum-of-ratios fractional optimization problems with block structure. This model class is ubiquitous and encompasses several important nonsmooth optimization problems in the literature. We first propose an inertial proximal block coordinate method for solving this class of problems by exploiting the underlying structure. The global convergence of our method is guaranteed under the Kurdyka--{\L}ojasiewicz (KL) property and some mild assumptions. We then identify the explicit exponents of the KL property for three important structured fractional optimization problems. In particular, for the sparse generalized eigenvalue problem with either cardinality regularization or sparsity constraint, we show that the KL exponents are $1/2$, and so, the proposed method exhibits linear convergence rate. Finally, we illustrate our theoretical results with both analytic and simulated numerical examples.
\end{abstract}

\begin{keywords}
fractional program, 
Kurdyka--{\L}ojasiewicz property, 
linear convergence, 
proximal block coordinate method,
sparsity, 
sum-of-ratios
\end{keywords}

\begin{AMS}
90C26,	
90C32,	
49M27,	
65K05	
\end{AMS}

\section{Introduction}

We consider the following nonsmooth and nonconcave fractional maximization problem
\begin{equation}\label{eq:prob}
\max_{\x =(x_1,\dots,x_m)\in S:=S_1 \times \dots \times S_m} F(\x) :=h(x_1,\dots,x_m)+\sum_{i=1}^m \frac{f_{i}(x_i)}{g_{i}(x_i)},
\tag{\ensuremath{\mathcal{P}}}
\end{equation}
where, for each $i\in \{1,\dots,m\}$, $\mathcal{H}_i$ is a finite-dimensional real Hilbert space, $S_i$ is a nonempty closed subset of $\mathcal{H}_i$, $h\colon \mathcal{H}_1 \times \dots \times \mathcal{H}_m \to \left[-\infty,+\infty\right]$ is a (possibly) nonsmooth and nonconcave function, and $f_i,g_i\colon \mathcal{H}_i\to \mathbb{R}$ are locally Lipschitz functions such that, for all $x_i\in S_i$,
\begin{equation}\label{eq:nonnegative}
f_{i}(x_i) \geq 0 \quad\text{and}\quad g_{i}(x_i) > 0.
\end{equation}
The model problem~\eqref{eq:prob} covers various important optimization problems arising in diverse areas, such as the energy efficiency maximization problem and the sparse generalized eigenvalue problem \cite{TWLZ18,ZD15}. On the other hand, it belongs to the class of so-called sum-of-ratios optimization problems which are known as the most difficult problems in the fractional programming literature. Obviously, there is an alternative formulation for \eqref{eq:prob} which is obtained by replacing the maximum with minimum. Although these two formulations are in general independent (due to the nonnegativity assumption \eqref{eq:nonnegative}), the corresponding algorithmic development can be easily modified to suit the other formulation. Therefore, in this paper, we will focus on the maximum formulation. Below, we present a few explicit motivating examples for the model problem~\eqref{eq:prob}.

\begin{example}[penalization  formulation for general sum-of-ratios optimization problem] \label{ex:General_sum_of_ratio}
Consider the classical sum-of-ratios optimization problem
\begin{equation}\label{eq:classical}
\max_{z\in C} \sum_{i=1}^m \frac{f_i(z)}{g_i(z)},  
\end{equation}
where $C$ is a bounded polyhedron in $\mathbb{R}^d$ and, for each $i\in \{1,\dots,m\}$, $f_i$ and $g_i$ are continuously differentiable functions on $\mathbb{R}^d$ such that, for all $z \in C$, $f_i(z) \geq 0$ and $g_i(z) >0$. This, for example, covers the energy efficiency maximization problem discussed in \cite{ZD15}, where $C =\{z \in \mathbb{R}_+^{d}: \forall i\in \{1,\dots,m\},\ z_i^{\min} \leq z_i \leq z_i^{\max} \text{~and~} v^\top z \leq r\}$ with $0 < z_i^{\min} \leq z_i^{\max}$, $v \in \mathbb{R}^d_+$, $r>0$ and, for each $i\in \{1,\dots,m\}$, $f_i(z)=\log(1+u_i^\top z+r_i)$ with $u_i \in \mathbb{R}_+^{d}\smallsetminus\{0\}$ and $r_i \geq 0$, and $g_i$ is an affine function with positive values on $C$. Note that \eqref{eq:classical} 
can be equivalently rewritten as
\begin{equation*}
\max_{x_1,\dots, x_m\in C} \sum_{i=1}^m \frac{f_i(x_i)}{g_i(x_i)} \quad\text{s.t.}\quad  x_1 =\dots =x_m.    
\end{equation*}
Therefore, a plausible alternative optimization formulation for \eqref{eq:classical} becomes
\begin{equation}\label{eq:classical0}
\max_{x_1,\dots, x_m\in C}  -\gamma \sum_{i=2}^{m}\|x_1-x_i\|^2+\sum_{i=1}^m \frac{f_i(x_i)}{g_i(x_i)},     
\end{equation}
where $\gamma>0$ is a parameter. Direct verification shows that this is a particular case of our model problem~\eqref{eq:prob} with $h(x_1,\dots, x_m) =-\gamma \,\sum_{i=2}^{m}\|x_1-x_i\|^2$, and $S =S_1\times \dots\times S_m$ with $S_i =C$, $i =1,\dots,m$.
\end{example}

\begin{example}[sparse generalized eigenvalue problem]
\label{ex:eigen}
The generalized eigenvalue problem, which searches for the most dominant eigenvalues (or principal eigenvalues) and corresponding eigenvector, can be written as an optimization problem $\max_{\x \in \mathbb{R}^d} \{\frac{\x^\top A\x}{\x^\top B\x}: \|\x\|=1\}$. In numerical analysis, one seeks an eigenvector with least number of nonzero entries, so that the information can be easily stored, explained and identified. This leads to a sparse generalized eigenvalue problem which can be formulated as
\begin{equation*}
\max_{\x \in \mathbb{R}^d} \frac{\x^\top A\x}{\x^\top B\x} -\lambda \phi(\x)
\quad\text{s.t.}\quad \|\x\|=1.
\end{equation*}
Here, $A,B$ are symmetric matrices with $A$ positive semidefinite and $B$ positive definite, and $\phi$ is a regularization function which induces sparsity of the solution. Typical choices of $\phi$ include the $\ell_0$ regularization (or cardinality) function given by $\|\x\|_0 =\{\text{number of~} i: x_i \neq 0\}$,  the $\ell_1$-norm given by $\|\x\|_1 =\sum_{i=1}^d |x_i|$, and the indicator function of the sparsity set $C_r =\{\x\in \mathbb{R}^d: \|\x\|_0 \leq r\}$ with $r>0$.  For example, in a recent study \cite{TWLZ18}, the authors examined the sparse generalized eigenvalue problem with $\phi(\x) =\delta_{C_r}(\x)$, where they proposed a truncated Rayleigh flow method (TRFM) and demonstrated the efficiency of this model problem on classification, correlation analysis and regression. Direct verification shows that the sparse generalized eigenvalue problem is a particular case of our model problem~\eqref{eq:prob} with $m =1$, $h(\x) =-\lambda \phi(\x)$, $f_1(\x) =\x^\top A\x$, $g_1(\x) =\x^\top B\x$, and $S =\{\x\in \mathbb{R}^d: \|\x\|=1\}$.
\end{example}

\begin{example}[maximizing the sum of a quadratic function and the Rayleigh quotient over the unit
sphere]
\label{ex:Rayleigh}
We consider the problem of maximizing the sum of a quadratic function and Rayleigh quotient over the unit
sphere
\begin{equation*}
\max_{\x \in \mathbb{R}^d} \x^\top W\x+ \frac{\x^\top A\x}{\x^\top B\x} 
\quad\text{s.t.}\quad \|\x\|=1,
\end{equation*}
where $A,B$ are positive definite matrices and $W$ is a symmetric matrix.
This problem arises in sparse Fisher discriminant analysis, in the context of which it is usually solved iteratively \cite{Zha13}. In particular, $\x$ is the desired discriminating vector in cluster analysis, the term $\frac{\x^\top A\x}{\x^\top B\x}$ is known as the Rayleigh quotient (or Fisher information in information science), and the quadratic term $\x^\top W\x$ serves as a local approximation for the sparse penalty term. Direct verification shows that this is a particular case of our model problem~\eqref{eq:prob} for $m =1$, $h(\x) =\x^\top W\x$, $f_1(\x) =\x^\top A\x$, $g_1(\x) =\x^\top B\x$ and $S =\{\x\in \mathbb{R}^d: \|\x\|=1\}$.
\end{example}

In the case where $m=1$ and $h \equiv 0$, problem~\eqref{eq:prob} is known as the single ratio fractional programming problem $\max_{\x \in S}\frac{f_1(\x)}{g_1(\x)}$. A classical approach for solving the latter problem is Dinkelbach's method and its variants (see \cite{CFS85,Din67}). In this approach, one typically constructs an iterative scheme which requires finding an optimal solution $\x_{n+1}$ of the optimization problem
\begin{equation}\label{eq:Dinkelbach}
\max_{\x \in S} \{f_1(\x) -\theta_n g_1(\x)\}
\end{equation}
in each iteration $n$, while $\theta_n$ is updated by $\theta_{n+1} :=\frac{f_1(\x_{n+1})}{g_1(\x_{n+1})}$. For details of this approach, we refer the readers to  \cite{CFS85,Din67,Iba83,Sch75}. However, solving in each iteration an optimization problem of type \eqref{eq:Dinkelbach} may be as expensive and difficult as solving the original problem in general. Recently, proximal type methods based on Dinkelbach's approach have been proposed to tackle single ratio fractional programs \cite{BC17,BDL20,LSZZ22}, where each subproblem is much easier to solve and sometimes has closed form solutions.

Unfortunately, in the case of sum-of-ratios fractional programs, that is either $m>1$ or $h \not\equiv 0$, Dinkelbach's approach cannot be directly applied anymore. One naive approach is to convert the sum-of-ratios into single ratio's cases and to apply Dinkelbach's method. This approach increases the complexity of the function dramatically and leads to numerical methods with poor performance.  For example, through this approach, a sum of three linear fractional functions becomes a fractional function whose numerator and denominator are  degree 3 nonconvex polynomials, and so, the nice linearity structure is completely lost. Some important steps towards solving sum-of-ratios fractional programs are mainly limited to sum-of-ratios of linear or quadratic fractional programs, and rely on  integer programming techniques such as branch and bound and convex relaxation methods, see for example \cite{Ben04,Loc15,Zha13}.
These approaches, although highly appealing, are much less scalable than the proximal type methods, and cannot directly deal with settings in which nonsmooth functions are involved.

Despite this important progress, it is still no clear whether one can develop proximal methods for solving nonsmooth and nonconcave sum-of-ratios fractional programs \eqref{eq:prob} in the line of \cite{BC17,BDL20} for single ratio cases. This forms the basic motivation of our work. Specifically, the contributions of this paper are as follows:
\begin{enumerate}[label =(\arabic*)]
\item
In Section~\ref{s:main}, we propose an inertial proximal subgradient method for solving the model problem~\eqref{eq:prob}. We then show that the iterative sequence generated by the proposed method is bounded and any of its limit points is a stationary point of problem~\eqref{eq:prob} in a suitable sense. This new method can be interpreted as a proximal block coordinate method of Gauss--Seidel type applied to a related non-fractional reformulated problem.  
We also establish the convergence of the full sequence under the assumption that a suitable merit function satisfies the Kurdyka--{\L}ojasiewicz (KL) property.

\item
In Section~\ref{s:KL}, we analyze several structured sum-of-ratios fractional programs and obtain the explicit KL exponents of the corresponding desingularization functions in the KL property: sum-of-ratios fractional quadratic programs with spherical constraint, generalized eigenvalue problems with cardinality regularization and generalized eigenvalue problems with sparsity constraints. In particular, we establish that, for the last two classes of fractional programs, the KL exponents are $1/2$. As a consequence, we obtain that the proposed numerical method exhibits linear convergence for these two classes of fractional programs.

\item
Finally, we illustrate the proposed method via both analytical and simulated numerical examples in Section~\ref{s:numerical}.
\end{enumerate}

\section{Preliminaries}

In this section, we recall some basic notations and preliminary results which will be used in this paper. We assume throughout that $\mathcal{H}$, $\mathcal{H}_1$, \dots, $\mathcal{H}_m$ are finite-dimensional real Hilbert spaces
with inner product $\scal{\cdot}{\cdot}$ and induced norm $\|\cdot\|$. The product space $\mathcal{H}_1 \times \dots \times \mathcal{H}_m$ is also a real Hilbert space endowed with the inner product given by
    $\langle (x_1,\dots,x_m),(y_1,\dots,y_m) \rangle = \sum_{i=1}^m \langle x_i,y_i \rangle.$
The set of nonnegative integers is denoted by $\mathbb{N}$, the set of real numbers by $\mathbb{R}$, the set of nonnegative real numbers by $\RP$, and the set of the positive real numbers by $\RPP$.

The \emph{indicator function} of a set $C$ is defined by $\delta_C(\x) :=0$ if $\x \in C$, and $\delta_C(\x) :=+\infty$ otherwise. Given an extended-real-valued function $f\colon \mathcal{H}\to \left[-\infty,+\infty\right]$, its \emph{domain} is defined by $\dom f :=\{\x \in \mathcal{H}: f(\x)<+\infty\}$. The function $f$ is \emph{proper} if $\dom f\neq \varnothing$ and it never equals $-\infty$. We say that $f$ is \emph{lower semicontinuous} if, for all $\x\in  \mathcal{H}$, $f(\x)\leq \liminf_{\y\to \x} f(\y)$, and \emph{convex} if its epigraph $\{(\x,\rho)\in \mathcal{H}\times \mathbb{R}: f(\x)\leq \rho\}$ is a convex subset of $\mathcal{H}\times \mathbb{R}$. The function $f$ is said to be \emph{weakly convex (on $\mathcal{H}$)} if there exists $\alpha \geq 0$ such that $f+\frac{\alpha}{2}\|\cdot\|^2$ is a convex function. The smallest constant $\alpha$ such that $f+\frac{\alpha}{2}\|\cdot\|^2$ is convex is called the \emph{modulus} of weak convexity. More generally, $f$ is said to be \emph{weakly convex on $S \subseteq \mathcal{H}$ with modulus $\alpha$} if $f+\delta_S$ is weakly convex with modulus $\alpha$. Weakly convex functions form a broad class of functions which covers quadratic functions, convex functions, differentiable functions whose gradient is Lipschitz continuous, and the composition of a convex and Lipschitz continuous function with a $C^1$-smooth mapping whose Jacobian is Lipschitz continuous (see \cite[Lemma~4.2]{DP19}).

Let $f\colon\mathcal{H} \to \left[-\infty,+\infty\right]$ and $\x\in \mathcal{H}$ with $|f(\x)| <+\infty$. The \emph{Fr\'echet subdifferential} of $f$ at $\x$ is given by
\begin{equation*}
\widehat{\partial} f(\x) :=\left\{\bu\in \mathcal{H}:\; \liminf_{\z\to \x}\frac{f(\z)-f(\x)-\langle \bu,\z-\x\rangle}{\|\z-\x\|}\geq 0\right\},
\end{equation*}
the \emph{limiting subdifferential} of $f$ at $\x$ is given by
\begin{equation*}
\partial_L f(\x) :=\left\{\bu\in \mathcal{H}:\; \exists \x_n\st{f}{\to}\x,\ \bu_n\to \bu\text{~~with~~} \bu_n\in \widehat{\partial} f(\x_n)\right\},
\end{equation*}
and the \emph{horizon subdifferential} of $f$ at $\x$ is given by
\begin{equation*}
\partial^{\infty}f(\x) :=\left\{\bu\in \mathcal{H}:\; \exists \x_n\st{f}{\to}\x,\ \lambda_n\downarrow 0,\ \lambda_n\bu_n\to \bu\text{~~with~~} \bu_n\in \widehat{\partial} f(\x_n)\right\}.
\end{equation*}
Here, the notation $\z\st{f}{\to}\x$ means $\z\to \x$ with $f(\z)\to f(\x)$. It follows from the above definition that the limiting subdifferential has the following \emph{robustness property}
\begin{equation*}
\partial_L f(\x) =\left\{\bu\in \mathcal{H}:\; \exists \x_n \st{f}{\to}\x,\ \bu_n\to \bu \text{~~with~~} \bu_n\in \partial_L f(\x_n)\right\}.
\end{equation*}
The \emph{domain} of $\partial_L f$ is $\dom\partial_L f :=\{\x\in \mathcal{H}: \partial_L f(\x)\neq \varnothing\}$. If $f$ is Lipschitz continuous around $\x$, then $\partial_L f(\x)$ is bounded and $\partial^{\infty} f(\x) =\{0\}$ (see \cite[Corollary~1.81]{Mor06}). If $f$ is strictly differentiable\footnote{A function $f$ is strictly differentiable at $\x$ if there exists $\bu \in \mathcal{H}$ such that $\lim_{\y,\z \to \x}\frac{f(\y)-f(\z)-\langle \bu,\y-\z\rangle}{\|\y-\z\|} =0$. Clearly, if $f$ is continuously differentiable at $\x$, then it is strictly differentiable at $\x$.} at $\x$, then $\widehat{\partial} f$ and $\partial_L f$ reduce to the derivative of $f$, denoted by $\nabla f$ (see \cite[Corollary~1.82]{Mor06}). If $f$ is convex, then both Fr\'echet and limiting subdifferentials at $\x$ reduce to the classical subdifferential in convex analysis (see \cite[Theorem~1.93]{Mor06})
\begin{equation*}
\partial f(\x) :=\left\{\bu\in \mathcal{H}:\; \forall \z\in \mathcal{H},\; \langle \bu,\z-\x\rangle\leq f(\z)-f(\x)\right\}.
\end{equation*}
We say that $f$ is \emph{regular\footnote{This is also referred as \emph{lower regular} in \cite{Mor06,MNY06}.} at $\x\in \mathcal{H}$} if $\widehat{\partial} f(\x) =\partial_L f(\x)$, and that $f$ is \emph{regular on $C\subseteq \mathcal{H}$} if it is regular at any $\x\in C$. For a  proper lower semicontinuous function $f$, it is clear that if $f$ is convex around $\x$ or strictly differentiable at $\x$, then it is regular at $\x$. A nonempty set $S$ in $\mathcal{H}$ is \emph{regular} at $\x\in S$ if $\delta_S$ is regular at $\x$. We say that $S$ is regular if it is regular at all of its points. From the definition, it can be verified that $C$ is regular if $C$ is a closed and convex set or $C$ is a smooth manifold given by $C =\{\x\in \mathcal{H}: g_i(\x)=0,\, i=1,\dots, m\}$, where $g_i$ are smooth functions satisfying the so-called linear independent constraint qualification (that is, $\{\nabla g_i(\x): i=1,\dots,m\}$ are linearly independent for all $\x \in C$).

Next, we collect some subdifferential rules and calculations which will be of use in our analysis and whose proofs are given in Appendix~\ref{s:appendix}.

\begin{lemma}[calculus rules]
\label{l:calrules}
Let $f,g\colon \mathcal{H}\to \left(-\infty,+\infty\right]$ be proper lower semicontinuous functions and let $\x \in \dom f$. Then the following statements hold:
\begin{enumerate}
\item\label{l:calrules_separable}
(Separable sum rule) If $f(\x) =\sum_{i=1}^m f_i (x_i)$ with $\x =(x_1,\dots,x_m)$, then $\partial_L f(\x) =\partial_L f_1(x_1) \times \dots \times \partial_L f_m(x_m)$ and $f$ is regular at $\x$ when each $f_i$ regular at $x_i$.

\item\label{l:calrules_sum}
(Sum rule)
If $\partial^{\infty} f(\x) \cap (-\partial^{\infty} g(\x)) =\{0\}$, then $\partial_L(f+g)(\x) \subseteq \partial_L f(\x) + \partial_L g(\x)$, where the equality holds when both $f$ and $g$ are regular at $\x$, in which case $f+g$ is also regular at $\x$. Moreover, if $g$ is strictly differentiable at $\x$, then $\partial_L (f+g)(\x) =\partial_L f(\x) +\nabla g(\x)$.

\item\label{l:calrules_sign}
(Sign rule)
If $f$ is Lipschitz continuous around $\x$ and $\widehat{\partial} f$ is nonempty-valued around $\x$, then $\partial_L (-f)(\x)\subseteq -\partial_L f(\x)$. 

\item\label{l:calrules_quotient}
(Quotient rule) 
Suppose that $f$ and $g$ are Lipschitz continuous around $\x$, and $g(\x)\neq 0$. If $\widehat{\partial} f$ is nonempty-valued around $\x$, then
\begin{equation*}
\partial_L\left(\frac{-f}{g}\right)(\x) \subseteq \frac{-g(\x)\partial_L f(\x) +\partial_L(f(\x)g)(\x)}{g(\x)^2}.
\end{equation*}
If $f$ is strictly differentiable at $\x$, then
\begin{equation*}
\partial_L\left(\frac{-f}{g}\right)(\x) =\frac{-g(\x)\nabla f(\x) +\partial_L(f(\x)g)(\x)}{g(\x)^2}
\end{equation*}
and, consequently, $-f/g$ is regular at $\x$ if and only if $f(\x)g$ is regular at $\x$.

\item\label{l:calrules_sqrt}
(Chain rule and square root rule)
If $f$ is Lipschitz continuous around $\x$ and $\theta\colon \mathbb{R}\to \mathbb{R}$ is continuously differentiable around $f(\x)$,  
then $\partial_L(\theta \circ f)(\x) =\partial_L (\theta'(f(\x))f)(\x)$. In particular, if $f$ is Lipschitz continuous around $\x$ and $f(\x)>0$, then
\begin{equation*}
\partial_L\left(\sqrt{f}\right)(\x) =\frac{\partial_L f(\x)}{2 \sqrt{f(\x)}} \text{~~and~~} \partial_L\left(-\sqrt{f}\right)(\x) =\frac{\partial_L (-f)(\x)}{2 \sqrt{f(\x)}}.
\end{equation*}
\end{enumerate}
\end{lemma}

\begin{lemma}
\label{l:subdiff}
Let $\Lambda :=\{\x\in \mathbb{R}^d: \|\x\|=1\}$ and $C_r :=\{\x\in \mathbb{R}^d: \|\x\|_0 \leq r\}$. Given $\x =(x_1,\dots,x_d)\in \mathbb{R}^d$, set $\supp(\x): =\{j: x_j \neq 0\}$. Then the following statements hold:
\begin{enumerate}
\item\label{l:subdiff_l0}
$\forall \x\in \mathbb{R}^d$, $\widehat{\partial} (\|\cdot\|_0)(\x) =\partial_L (\|\cdot\|_0)(\x) =\partial_L^{\infty} (\|\cdot\|_0)(\x) =\{\bv: v_j=0 \text{~if~} j \in \supp(\x)\}$.

\item\label{l:subdiff_sphere}
$\forall \x\in \Lambda$, $\widehat{\partial} \delta_{\Lambda}(\x) =\partial_L \delta_{\Lambda}(\x) =\partial_L^{\infty} \delta_{\Lambda}(\x) =\{ t\x: t \in \mathbb{R}\}$.

\item\label{l:subdiff_ball}
If $\|\x\|_0=r$, then $\partial_L^{\infty} \delta_{C_r}(\x) =\partial_L \delta_{C_r}(\x) =\{\bv: v_j=0 \text{~if~} j \in \supp(\x)\}$. If $\|\x\|_0<r$, then
\begin{multline*}
\partial_L^{\infty} \delta_{C_r}(\x) =\partial_L \delta_{C_r}(\x) =\{\bv: \exists \widehat{J} \subseteq \{1,\dots,d\} \smallsetminus \supp(\x) \text{~with~} |\widehat{J}|=r-\|\x\|_0,\\ v_j=0 \text{~if~} j \in (\supp(\x) \cup \widehat{J})\}.    
\end{multline*}

\item\label{l:subdiff_sum}
$\forall \x\in \Lambda$, $\partial_L (\|\cdot\|_0+\delta_{\Lambda})(\x) =\partial_L (\|\cdot\|_0)(\x) +\partial_L \delta_{\Lambda}(\x)$.

\item\label{l:subdiff_sum'}
$\forall \x\in \Lambda \cap C_r$, $\partial_L (\delta_{C_r}+\delta_{\Lambda})(\x)\subseteq \partial_L\delta_{C_r}(\x) +\partial_L \delta_{\Lambda}(\x)$.
\end{enumerate}
\end{lemma}

We also need the following two notions of stationary points for problem \eqref{eq:prob}.
\begin{definition}[stationary points]
We say that $\overline{\x} =(\overline{x}_1, \dots, \overline{x}_m) \in S$ is a \emph{stationary point} for \eqref{eq:prob} if $0 \in \partial_L (-F+\delta_S)(\overline{\x})$, and a \emph{lifted coordinate-wise stationary point} for \eqref{eq:prob} if, for each $i\in \{1,\dots,m\}$,
\begin{equation*}
0\in \partial_L^{x_i} (-h+\delta_S)(\overline{\x}) +\frac{ -g_i(\overline{x}_i)\partial_L f_i(\overline{x}_i) +f_i(\overline{x}_i)\partial_L g_i(\overline{x}_i)}{g_i(\overline{x}_i)^2},
\end{equation*} 
where $\partial_L^{x_i}$ denotes the subdifferential with respect to the $x_i$-variable.
\end{definition}

The following lemma whose proof is given in Appendix~\ref{s:appendix} provides the relationship between a stationary point and a lifted coordinate-wise stationary point for \eqref{eq:prob}. 

\begin{lemma}[stationary vs. lifted coordinate-wise stationary points]
\label{l:stationary}
Let $\overline{\x} =(\overline{x}_1, \dots, \overline{x}_m) \in \mathcal{H}_1\times \dots\times \mathcal{H}_m$. Suppose that $-h$ is proper lower semicontinuous, that either $m =1$ or $h$ is strictly differentiable at $\overline{\x}$, and that, for each $i\in \{1,\dots,m\}$, $f_i$ and $g_i$ are Lipschitz continuous around $\overline{x}_i$, $f_i(\overline{x}_i) \geq 0$, and $g_i(\overline{x}_i) \neq 0$. Then the following statements hold:
\begin{enumerate}
\item\label{l:stationary_imply}
If for each $i\in \{1,\dots,m\}$, $\widehat{\partial} f_i$ is nonempty-valued around $\overline{x}_i$, then $\overline{\x}$ is a lifted coordinate-wise stationary point for \eqref{eq:prob} whenever it is a stationary point for \eqref{eq:prob}.

\item\label{l:stationary_equi}
If for each $i\in \{1,\dots,m\}$, $f_i$ is strictly differentiable at $\overline{x}_i$ and either (a)
$g_i$ is strictly differentiable at $\overline{x}_i$ or (b) $g_i$ is regular at $\overline{x}_i$ and $-h+\delta_S$ is regular at $\overline{\x}$, then $\overline{\x}$ is a lifted coordinate-wise stationary point for \eqref{eq:prob} if and only if it is a stationary point for \eqref{eq:prob}.
\end{enumerate}
\end{lemma}

\section{Inertial proximal subgradient method}
\label{s:main}

In this section, we propose an inertial proximal subgradient method for solving the sum-of-ratios optimization problem \eqref{eq:prob} and establish the convergence analysis for the proposed method. From now on, we will work under the following assumption.

\begin{assumption}
\label{a:standing}
For problem~\eqref{eq:prob}, $S$ is a (not necessarily convex) closed set, $-h$ is a proper lower semicontinuous function, and, for each $i\in \{1,\dots,m\}$, the functions $f_i$ and $g_i$ are locally Lipschitz functions on an open set containing $S_i$. Moreover,
\begin{enumerate}[label =(\alph*)]
\item\label{a:f}
For each $i\in \{1, \dots, m\}$, $f_i$ is nonnegative on an open set containing $S_i$ and there exists $\alpha_i \geq 0$ such that, for all $x_i, z_i \in S_i$ and all $u\in \partial_L f_i(x_i)$,
\begin{equation*}
\scal{\frac{u}{2\sqrt{f_i(x_i)}}}{z_i-x_i} \leq \sqrt{f_i(z_i)} -\sqrt{f_i(x_i)} +\frac{\alpha_i}{2}\|z_i-x_i\|^2,
\end{equation*}
whenever $f_i(x_i) >0$.

\item\label{a:g}
For each $i\in \{1, \dots, m\}$, $g_i$ is positive on $S_i$ and there exists $\beta_i\geq 0$ such that, for all $x_i, z_i\in S_i$ and all $v\in \partial_L g_i(x_i)$,
\begin{equation*}
\scal{v}{z_i-x_i} \geq g_i(z_i) -g_i(x_i) -\frac{\beta_i}{2}\|z_i-x_i\|^2.
\end{equation*}
\end{enumerate}
\end{assumption}

\begin{remark}[comments for the standing assumption]
We note that standing Assumption~\ref{a:standing} is quite general and, in particular, are satisfied for our motivating examples.
\begin{enumerate}
\item 
Assumption~\ref{a:standing}\ref{a:f} is fulfilled if, for each $i\in \{1,\dots,m\}$, $f_i$ takes nonnegative values  on an open set $\mathcal{O}_i$ containing $S_i$ and $\sqrt{f_i}$ is weakly convex on $\mathcal{O}_i$ with modulus $\alpha_i$. Clearly, this condition is true if $f_i(x_i) =x_i^\top A_ix_i$ for a positive semi-definite matrix $A_i$ (as in the motivating Example~\ref{ex:Rayleigh} and Example~\ref{ex:eigen}) because $\sqrt{f_i}(x_i)=\|A_i^{1/2}x_i\|$ which is convex, where $A_i^{1/2}$ is a symmetric matrix such that $A_i^{1/2}A_i^{1/2} =A_i$.

Assumption~\ref{a:standing}\ref{a:f} also holds if, for each $i\in \{1,\dots,m\}$, $S_i$ is compact, $f_i$ takes positive values  on an open set $\mathcal{O}_i$ containing $S_i$ and $f_i$ is a differentiable function whose gradient is Lipschitz continuous on $\mathcal{O}_i$ with modulus $L_i$. Indeed, in this case, letting $r_i:=\min_{x_i \in S_i}f_i(x_i)>0$,  a direct verification shows that  $\sqrt{f}_i$ is weakly convex with modulus  $\alpha_i=\frac{L_i}{2\sqrt{r_i}}+\frac{1}{4}r_i^{\frac{-3}{2}}\max_{x_i \in S_i}\|\nabla f_i(x_i)\|^2$. This covers, in particular, the alternative optimization formulation for the energy maximization problem mentioned in the motivating Example~\ref{ex:General_sum_of_ratio}, where $f_i(x_i)=\log(1+u_i^\top x_i+r_i)$ with $u_i \in \mathbb{R}_+^{d_i}\smallsetminus\{0\}$ and $r_i \geq 0$ and, for $i\in \{1, \dots, m\}$, 
$S_i=\{x \in \mathbb{R}_+^{d}: x_i^{\min} \leq x_i \leq x_i^{\max} \mbox{ and } v^\top x \leq r\}$ with $0 < x_i^{\min} \leq x_i^{\max}$, $v \in \mathbb{R}^d_+$ and $r>0$.

Similarly, Assumption~\ref{a:standing}\ref{a:g} is satisfied if, for each $i\in \{1,\dots,m\}$, $g_i$ is positive on $S_i$ and it is a differentiable function whose gradient is Lipschitz continuous with modulus $\beta_i$.
Thus, combining these observations, we see  Assumption~\ref{a:standing} are satisfied for  the important  motivating examples mentioned in the introduction.

\item 
We also notice that the first condition in Assumption~\ref{a:standing}\ref{a:f} ensures that, if $x_i\in S_i$ and $f_i(x_i) =0$, then $0\in \partial_L f_i(x_i)$ for $i\in \{1, \dots, m\}$. 
\end{enumerate}
\end{remark}

We now propose our inertial proximal subgradient method for \eqref{eq:prob}. As we will see later on, this method can be seen as a proximal block coordinate method of Gauss--Seidel type applied to an equivalent non-fractional formulation. It is also worthwhile noting that, even when applied to the single-ratio case ($m =1$ and $h\equiv 0$), our method here is totally different from the proximal type methods in \cite{BC17,BDL20} which are based on Dinkelbach's approach.

\begin{tcolorbox}[
	left=0pt,right=0pt,top=0pt,bottom=0pt,
	colback=blue!10!white, colframe=blue!50!white,
  	boxrule=0.2pt,
  	breakable]
\begin{algo}[Inertial proximal subgradient method for problem \eqref{eq:prob}]
\label{algo:epasa}
\step{}
Choose $\x_{-1} =\x_0 =(x_{1,0},\dots,x_{m,0})\in S$ and set $n =0$. Let $\delta\in \RPP$ and $\overline{\nu}\in \left[0,\delta/2\right)$.

\step{}\label{step:main}
Set $\y_n =(y_{1,n}, \dots, y_{m,n})$ with $y_{i,n} =\frac{\sqrt{f_i(x_{i,n})}}{g_i(x_{i,n})}$. Choose $\tau_n\in \mathbb{R}$ such that $\tau_n \geq \delta +\max_{1\leq i\leq m} \{\frac{1}{2}(2y_{i,n}\alpha_i+y_{i,n}^2\beta_i)\}$, where $\alpha_i$ and $\beta_i$ are defined in Assumption~\ref{a:standing}.
Let $\nu_n\in \left[0,\overline{\nu}/\tau_n\right]$. For each $i\in \{1,\dots,m\}$, let $z_{i,n} =x_{i,n} +\nu_n(x_{i,n}-x_{i,n-1})$, $u_{i,n} \in \partial_L f_i(x_{i,n})$, $v_{i,n}\in \partial_L g_i(x_{i,n})$, and set
\begin{equation*}
w_{i,n} =\begin{cases}
y_{i,n}\frac{u_{i,n}}{\sqrt{f_i(x_{i,n})}} -{y}_{i,n}^2 v_{i,n} & \text{if~} f_i(x_{i,n})>0, \\
0  & \text{if~} f_i(x_{i,n})=0.
\end{cases}
\end{equation*}
Denote $h_{i,n+1}(x_i) :=h(x_{1,n+1},\dots,x_{i-1,n+1}, x_i, x_{i+1,n},\dots,x_{m,n})$ and compute
\begin{equation*}
x_{i,n+1} =\argmax_{x_i \in S_i} \left\{h_{i,n+1}(x_i) -\tau_n \left\|x_i-z_{i,n}-\frac{1}{2\tau_n} w_{i,n}\right\|^2\right\}.
\end{equation*}
Update $\x_{n+1} =(x_{1,n+1},\dots,x_{m,n+1})$.

\step{}
If a termination criterion is not met, set $n =n+1$ and go to Step~\ref{step:main}.
\end{algo}
\end{tcolorbox}

\begin{remark}[discussion on the computational costs]
The major computation cost lies in the update of $\x_{n+1}$ in Step~\ref{step:main}. The update, for each $i\in \{1,\dots,m\}$,
\begin{equation*}
x_{i,n+1} =\argmax_{x_i \in S_i}\left\{h_{i,n+1}(x_i)- \tau_n \left\|x_i-\left(z_{i,n}+\frac{1}{2\tau_n} w_{i,n}\right)\right\|^2\right\}
\end{equation*}
is equivalent to computing the proximal operator\footnote{The \emph{proximal operator} of a function $\varphi$ at $x$ is defined by $\prox_{\varphi}(x) =\argmin_y \left\{\varphi(y) +\frac{1}{2}\|y-x\|^2\right\}$.} of $\frac{1}{2\tau_n}(-h_{i,n+1}+\delta_{S_i})$ at the point $z_{i,n}+\frac{1}{2\tau_n} w_{i,n}$. This can be done efficiently in many situations, for example, in the following cases:
\begin{enumerate}
\item
if $h \equiv 0$, then this reduces to the projection onto the set $S_i$ which, in many cases, has closed forms. This is the case when $S_i$ is a box, $S_i$ is a sphere or a ball, $S_i =\{x: \|x\|_0 \leq r\}$ for $r>0$, $S_i =\{x: \|x\|=1 \text{~and~} \|x\|_0 \leq r\}$ (as in the motivation Example~\ref{ex:eigen} of the sparse generalized eigenvalue problem with $\phi(\x)$ being the indicator function of the sparsity set) and $S_i =\{X \in \mathbb{R}^{p \times d}: X^\top X =I_d\}.$

\item
if $m=1$, $h(\x)= -\lambda \|\x\|_0$ or $h(\x)= -\lambda \|\x\|_1$ with $\lambda \geq 0$, and $S=\{\x: \|\x\|=1\}$ (as in the motivating Example~\ref{ex:eigen} of sparse generalized eigenvalue problem with $\phi(\x)$ being the cardinality regularization or $\ell_1$-regularization), then the resulting proximal operator can be simplified to $\argmin_{\x \in \mathbb{R}^d} \{\|\x+\ba\|^2 -\tau h(\x) : \|\x\|=1\}$ for some $\ba \in \mathbb{R}^d$ and $\tau \geq 0$. This can be further rewritten as $\argmin_{\x \in \mathbb{R}^d} \{\langle 2\ba,\x \rangle -\tau  h(\x) : \|\x\|=1\}$, which has a closed form solution (see \cite[Proposition~6]{SBP15} and \cite{LT13}).

\item
if $h$ is a (possibly) nonconvex quadratic function  and $S_i =\{x: \|x\|=1\}$ (as in the motivating Example~\ref{ex:Rayleigh}), then the resulting problem is a nonconvex quadratic programming problem with norm constraint which is known as the trust region problem. In this case, this problem can be solved efficiently, for example, by solving a related single generalized eigenvalue problem (see \cite{AINT17}). 

\item
if $h$ can be expressed as the maximum of finitely many concave quadratic functions, that is, $h(\x)=\max_{1 \leq r \leq p}\{\frac{1}{2}\x^\top A_r\x+a_r^\top \x+\alpha_r\}$, where each $A_r$ is negative semi-definite  (and so, $h_{i,n+1}$ can also be expressed in this form) and $S_i$ is a polyhedral set, then this is equivalent to the solving of $p$ many quadratic programming problems with linear inequality constraints, and so, it can be solved efficiently via quadratic programming solvers. This, in particular, covers the motivating Example~\ref{ex:General_sum_of_ratio}.
\end{enumerate}
Finally, we also remark that Step~\ref{step:main} also requires the availability of a subgradient of $f_i$ at the current iterate. In general, this requires $f_i$'s to have some specific structure. On the other hand, in many important applications, $f_i$ can be expressed as the maximum/minimum of finitely many continuously differentiable functions, in which case, a subgradient of $f_i$ is easily obtained. 
\end{remark}

\subsection{Interpretation of Algorithm~\ref{algo:epasa}}

Next, we see that  Algorithm~\ref{algo:epasa} can be interpreted as a proximal block coordinate method of Gauss--Seidel type applied to the problem
\begin{equation}\label{eq:prob1}
\max_{\substack{\x=(x_1, \dots, x_m)\in S \\ \y=(y_1, \dots, y_m)\in \mathbb{R}^m}} h(\x)+ H(\x,\y) \quad\text{with~}
H(\x,\y) :=\sum_{i=1}^m \left[2y_i\sqrt{f_i(x_i)}-y_i^2g_i(x_i)\right].
\tag{\ensuremath{\mathcal{P}_1}}
\end{equation}
We say that $(\overline{\x},\overline{\y})\in S\times \mathbb{R}^m$ is a \emph{lifted coordinate-wise stationary point} for \eqref{eq:prob1} if, for each $i\in \{1,\dots,m\}$, 
\begin{equation*}
0\in \partial_L^{x_i} (-h+\delta_S)(\overline{\x}) +\partial_L^{x_i} (-H)(\overline{\x},\overline{\y})
\text{~~and~~} \overline{y}_i =\sqrt{f_i(\overline{x}_i)}/g_i(\overline{x}_i),
\end{equation*}
where the latter is equivalent to $0\in \partial_L^{y_i} (-H)(\overline{\x},\overline{\y})$. 

The relationship between lifted coordinate-wise stationary points for \eqref{eq:prob} and \eqref{eq:prob1} is examined in the next lemma with proof in Appendix~\ref{s:appendix}. In the case where $h \equiv 0$, the following property of {\it global solutions} of \eqref{eq:prob} and \eqref{eq:prob1} was mentioned in \cite[Theorem~2.2]{Ben04} for problem \eqref{eq:classical} with affine numerators and denominators, and given in \cite[Corollary~1]{SY18} for problem \eqref{eq:classical} with non-affine numerators and denominators. It is worth noting that \cite{SY18} only provides the non-fractional reformulation for the nonconvex problem \eqref{eq:classical} in terms of global solutions, and the numerical algorithms were given only for concave-convex cases (that is, all the numerators are concave and denominators are convex, see \cite[Algorithm~1]{SY18}).  Unfortunately, the methods suggested  therein are not of the form of splitting algorithms, and there is no convergence guarantee  provided for the general setting in this paper covering the motivation examples in the introduction.

\begin{lemma}[fractional vs. non-fractional formulations]
\label{l:equi}
Let $\overline{\x} =(\overline{x}_1,\dots,\overline{x}_m) \in \mathcal{H}_1\times \dots \times \mathcal{H}_m$ and $\overline{\y} =(\overline{y}_1, \dots, \overline{y}_m) \in \mathbb{R}^m$ with
$\overline{y}_i =\frac{\sqrt{f_i(\overline{x}_i)}}{g_i(\overline{x}_i)}$. Then the following statements hold:
\begin{enumerate}
\item\label{l:equi_global}
$\overline{\x}$ is a global solution for \eqref{eq:prob} if and only if $(\overline{\x},\overline{\y})$ is a global solution for \eqref{eq:prob1}, in which case, both problems have the same optimal value.

\item\label{l:equi_local}
Suppose that $-h$ is proper lower semicontinuous and finite at $\overline{\x}$ and that, for each $i\in \{1,\dots,m\}$, $f_i$ and $g_i$ are Lipschitz continuous around $\overline{x}_i$, $f_i(\overline{x}_i) >0$, and $g_i(\overline{x}_i) >0$. Then 
\begin{enumerate}
\item
If for each $i\in \{1,\dots,m\}$, $\widehat{\partial} f_i$ is nonempty-valued around $\overline{x}_i$, then $\overline{\x}$ is a lifted coordinate-wise stationary point for $\eqref{eq:prob}$ whenever $(\overline{\x},\overline{\y})$ is a lifted coordinate-wise stationary point for \eqref{eq:prob1}.
\item 
If for each $i\in \{1,\dots,m\}$,  
$f_i$ is strictly differentiable at $\overline{x}_i$, then $\overline{\x}$ is a lifted coordinate-wise stationary point for $\eqref{eq:prob}$ if and only if $(\overline{\x},\overline{\y})$ is a lifted coordinate-wise stationary point for \eqref{eq:prob1}.
\end{enumerate}
\end{enumerate}
\end{lemma}

\begin{remark}[interpretation of Algorithm~\ref{algo:epasa} as a block coordinate inertial proximal algorithm]
Suppose that, for each $i\in \{1,\dots,m\}$, $f_i$ is nonnegative and $g_i$ is continuously differentiable on an open set containing $S_i$. We will show that Algorithm~\ref{algo:epasa} can be interpreted as a block coordinate inertial proximal  subgradient algorithm. To see this, we recall that, according to Lemma~\ref{l:equi}, problem~\eqref{eq:prob} is equivalent to \eqref{eq:prob1}. 
As, for each $i\in \{1,\dots,m\}$, $y_i \mapsto H_i(x_i,y_i) :=2y_i\sqrt{f_i(x_i)}-y_i^2g_i(x_i)$ is a strongly concave one-variable quadratic function which admits a global maximizer at $\frac{\sqrt{f_i(\x_{i})}}{g_i(\x_{i})}$, one has
\begin{align*}
\y_{n+1} &=\argmax_{\y \in \mathbb{R}^m}  \{h(\x_{n+1}) +\sum_{i=1}^m H_i(x_{i,n+1},y_{i})  \} \\
&=\argmax_{\y \in \mathbb{R}^m} \{h(\x_{n+1}) +H(\x_{n+1},\y)\}.
\end{align*}
Let $i\in \{1,\dots,m\}$. We see that,  if $f_i(x_{i,n})>0$, then
\begin{equation*}
w_{i,n} =y_{i,n}\frac{u_{i,n}}{\sqrt{f_i(x_{i,n})}} -y_{i,n}^2 v_{i,n}\in \frac{y_{i,n}\partial_L f_i(x_{i,n})}{\sqrt{f_i(x_{i,n})}} -y_{i,n}^2 \nabla g_i(x_{i,n}) =\partial_L^x H_i(x_{i,n},y_{i,n}).
\end{equation*}
If $f_i(x_{i,n}) =0$, then $y_{i,n}=0$, $w_{i,n} =0$, and, since $\sqrt{f_i(x_i)} \geq 0$ on an open set containing $S_i$ and $x_{i,n} \in S_i$, one has
$0 \in \partial_L(\sqrt{f_i})(x_{i,n})$, which implies that
\begin{equation*}
w_{i,n}=0 \in y_{i,n} \partial_L (\sqrt{f}_i)(x_{i,n}) - y_{i,n}^2 \nabla g_i(x_{i,n}) = \partial_L^x H_i(x_{i,n},y_{i,n}).
\end{equation*}
So, the update for $\x_{n+1} =(x_{1,n+1},\dots,x_{m,n+1})$ involves, for $i\in \{1,\dots,m\}$,
\begin{multline*}
x_{i,n+1} =\argmax_{x_i \in S_i}\left\{h_{i,n+1}(x_i)- \tau_n \left\|x_i-\left(z_{i,n}+\frac{1}{2\tau_n}w_{i,n} \right)\right\|^2\right\} \\ \text{~~with~~} w_{i,n}\in \partial_L^x H_i(x_{i,n},y_{i,n}).
\end{multline*}
Combining the above observations, one sees that Algorithm~\ref{algo:epasa} can be regarded as a block coordinate inertial proximal subgradient algorithm applied to problem~\eqref{eq:prob1}, where proximal subgradient steps are applied cyclically to the $\x$-variable and a direct maximization step is applied to the $\y$-variable.
\end{remark}

\subsection{Convergence analysis} 

In this part, we discuss the convergence analysis for Algorithm~\ref{algo:epasa}. Let us first start with the subsequential convergence. To do this, we shall consider the following assumption.
\begin{assumption}
\label{a:add}
For problem~\eqref{eq:prob}, either one of the following holds: 
\begin{enumerate}[label =(\alph*)]
\item
$m =1$, $-h +\delta_S$ and $g_1$ are regular on $S$, and $f_1$ is strictly differentiable on an open set containing $S$;
\item
$h$ is strictly differentiable on an open set containing $S$, $S$ is regular, and for each $i\in \{1,\dots,m\}$, $f_i$ is strictly differentiable on an open set containing $S_i$ and $g_i$ is regular on $S_i$;
\item
$h$ is strictly differentiable on an open set containing $S$ and, for each $i\in \{1,\dots,m\}$, $f_i$ and $g_i$ are strictly differentiable on an open set containing $S_i$.
\end{enumerate}
\end{assumption}

It is worth noting that Assumption \ref{a:add} is satisfied with all of our motivation  examples. We are now ready to state our first main result as below.

\begin{theorem}[subsequential convergence]
\label{t:cvg}
Let $(\x_n)_{n\in \mathbb{N}}$ be the sequence generated by Algorithm~\ref{algo:epasa}. Suppose that Assumption~\ref{a:standing} holds, that $F$ is bounded from above on $S$, and that the set $\{\x\in S: F(\x) \geq F(\x_0)\}$ is bounded. Then the following statements hold:
\begin{enumerate}
\item\label{t:cvg_decrease}
For all $n\in \mathbb{N}$, $F(\x_n) - \overline{\nu}\|\x_n-\x_{n-1}\|^2 \leq F(\x_{n+1}) -(\delta-\overline{\nu})\|\x_{n+1}-\x_n\|^2$.

\item\label{t:cvg_seq}
The sequence $(F(\x_n))_{n\in \mathbb{N}}$ is convergent, the sequence $(\x_n)_{n\in \mathbb{N}}$ is bounded, and $\sum_{n=0}^{+\infty} \|\x_{n+1}-\x_n\|^2 <+\infty$.

\item\label{t:cvg_crit}
Let $\overline{\x}$ be a cluster point of $(\x_n)_{n\in \mathbb{N}}$ and suppose that $\limsup_{n\to +\infty} \tau_n =\overline{\tau} <+\infty$ and that either $m =1$ or $h$ is continuous on $S \cap \dom h$. Then $\lim_{n\to +\infty} F(\x_n) =F(\overline{\x})$ and $\overline{\x}\in S$ is a lifted coordinate-wise stationary point for \eqref{eq:prob}. If additional Assumption~\ref{a:add} holds, then $\overline{\x}$ is a stationary point for \eqref{eq:prob}. 
\end{enumerate}
\end{theorem}
\begin{proof}
\ref{t:cvg_decrease}: 
Let any $i\in \{1,\dots,m\}$ and any $n\in \mathbb{N}$. From Step~\ref{step:main} of Algorithm~\ref{algo:epasa}, we have that
$\x_{i,n}\in S_i$, $y_{i,n}\geq 0$, and, for all $x_i \in S_i$,
\begin{equation*}
h_{i,n+1}(x_i)-\tau_n\left\|x_i-z_{i,n}-\frac{1}{2\tau_n}w_{i,n}\right\|^2 \leq h_{i,n+1}(x_{i,n+1})-\tau_n \left\|x_{i,n+1}-z_{i,n}-\frac{1}{2\tau_n}w_{i,n}\right\|^2,
\end{equation*}
which yields
\begin{align}
h_{i,n+1}(x_i) -h_{i,n+1}(x_{i,n+1}) &\leq -\tau_n\|x_{i,n+1}-z_{i,n}\|^2 +\tau_n\|x_i-z_{i,n}\|^2 +\langle w_{i,n},x_{i,n+1}-x_i \rangle \notag \\
&= -\tau_n\|x_{i,n+1}-x_{i,n}\|^2 +\tau_n\|x_i-x_{i,n}\|^2 +\langle w_{i,n},x_{i,n+1}-x_i \rangle \notag \\
&\hspace*{4cm} +2\tau_n\nu_n \langle x_{i,n+1}-x_i, x_{i,n}-x_{i,n-1}\rangle, \label{eq:use1}
\end{align}
where the last equality follows from the fact that $z_{i,n} =x_{i,n} +\nu_n(x_{i,n}-x_{i,n-1})$.
By letting $x_i =x_{i,n}$, 
\begin{align*}
h_{i,n+1}(x_{i,n}) -h_{i,n+1}(x_{i,n+1}) &\leq -\tau_n\|x_{i,n+1}-x_{i,n}\|^2 +\langle w_{i,n},x_{i,n+1}-x_{i,n} \rangle \\
&\qquad +2\tau_n\nu_n \langle x_{i,n+1}-x_{i,n}, x_{i,n}-x_{i,n-1}\rangle.
\end{align*}
Since $i$ is arbitrary and $2\scal{x_{i,n+1}-x_{i,n}}{x_{i,n}-x_{i,n-1}}\leq \|x_{i,n+1}-x_{i,n}\|^2 +\|x_{i,n}-x_{i,n-1}\|^2$, we deduce that
\begin{align}
h(\x_n) -h(\x_{n+1}) &= \sum_{i=1}^m \left(h_{i,n+1}(x_{i,n}) -h_{i,n+1}(x_{i,n+1})\right) \notag \\
&\leq -(\tau_n-\tau_n\nu_n)\|\x_{n+1}-\x_n\|^2 +\tau_n\nu_n\|\x_n-\x_{n-1}\|^2 \notag \\ 
&\qquad +\sum_{i=1}^m \scal{w_{i,n}}{x_{i,n+1}-x_{i,n}}, \label{eq:19}
\end{align}
where the first equality follows from the definition of $h_{i,n+1}$.

Next, we show that
\begin{equation*}
\scal{w_{i,n}}{x_{i,n+1}-x_{i,n}} \leq \frac{f_i(x_{i,n+1})}{g_i(x_{i,n+1})} -\frac{f_i(x_{i,n})}{g_i(x_{i,n})} +(\tau_n-\delta)\|x_{i,n+1}-x_{i,n}\|^2.
\end{equation*}
To see this, let us first consider the case when $f_i(x_{i,n}) >0$. Then $w_{i,n} 
=2y_{i,n}\frac{u_{i,n}}{2\sqrt{f_i(x_{i,n})}} -y_{i,n}^2v_{i,n}$.
Since $u_{i,n}\in \partial_L f_i(x_{i,n})$, the assumption on $\sqrt{f_i}$ gives
\begin{equation}\label{eq:sqrt f_i}
\scal{\frac{u_{i,n}}{2\sqrt{f_i(x_{i,n})}}}{x_{i,n+1}-x_{i,n}} \leq \sqrt{f_i(x_{i,n+1})} - \sqrt{f_i(x_{i,n})} +\frac{\alpha_i}{2}\|x_{i,n+1}-x_{i,n}\|^2.
\end{equation}
Since $v_{i,n}\in \partial_L g_i(x_{i,n})$, the assumption on $g_i$ gives
\begin{equation}\label{eq:g_i}
\scal{v_{i,n}}{x_{i,n+1}-x_{i,n}} \geq g_i(x_{i,n+1}) -g_i(x_{i,n}) -\frac{\beta_i}{2}\|x_{i,n+1}-x_{i,n}\|^2.
\end{equation}
Multiplying \eqref{eq:sqrt f_i} by $2y_{i,n} \geq 0$ and \eqref{eq:g_i} by $-y_{i,n}^2 \leq 0$ and then adding them we obtain that
\begin{multline}\label{eq:sumup}
\scal{w_{i,n}}{x_{i,n+1}-x_{i,n}} \leq H_i(x_{i,n+1},y_{i,n}) -H_i(x_{i,n},y_{i,n}) \\
+\frac{1}{2}(2y_{i,n}\alpha_i+y_{i,n}^2\beta_i)\|x_{i,n+1}-x_{i,n}\|^2,
\end{multline}
where $H_i(x_i,y_i) :=2y_i\sqrt{f_i(x_i)}-y_i^2g_i(x_i)$.
On the other hand, if $f_i(x_{i,n}) =0$, then $y_{i,n} = 0$ and $w_{i,n} =0$, hence \eqref{eq:sumup} still holds. In turn, from \eqref{eq:sumup} and the fact that $y_{i,n+1}$ is the maximizer of $H_i(x_{i,n+1},\cdot)$, we derive that
\begin{align*}
\scal{w_{i,n}}{x_{i,n+1}-x_{i,n}} &\leq H_i(x_{i,n+1},y_{i,n+1}) -H_i(x_{i,n},y_{i,n}) \\
&\qquad +\frac{1}{2}(2y_{i,n}\alpha_i+y_{i,n}^2\beta_i)\|x_{i,n+1}-x_{i,n}\|^2 \\
&\leq \frac{f_i(x_{i,n+1})}{g_i(x_{i,n+1})} -\frac{f_i(x_{i,n})}{g_i(x_{i,n})} +(\tau_n-\delta)\|x_{i,n+1}-x_{i,n}\|^2,
\end{align*}
where the last inequality follows by our choice of $\tau_n$. By combining this with \eqref{eq:19}, 
\begin{multline*}
(\delta -\tau_n\nu_n)\|\x_{n+1}-\x_n\|^2 -\tau_n\nu_n\|\x_n-\x_{n-1}\|^2 \\ \leq \left[h(\x_{n+1})+\sum_{i=1}^m \frac{f_i(x_{i,n+1})}{g_i(x_{i,n+1})}\right] - \left[h(\x_n)+\sum_{i=1}^m\frac{f_i(x_{i,n})}{g_i(x_{i,n})}\right].
\end{multline*}
Since $\nu_n \leq \overline{\nu}/\tau_n$, we get the claimed inequality.

\ref{t:cvg_seq}:
For all $n\in \mathbb{N}$, set $\theta_n :=F(\x_n) -\overline{\nu}\|\x_n-\x_{n-1}\|^2$. It follows from $\overline{\nu}\in \left[0,\delta/2\right)$ that $\delta-2\overline{\nu} >0$. According to \ref{t:cvg_decrease}, for all $n\in \mathbb{N}$,
\begin{equation}\label{eq:9}
\theta_n \leq \theta_{n+1} -(\delta-2\overline{\nu})\|\x_{n+1}-\x_n\|^2,
\end{equation}
and hence $(\theta_n)_{n\in \mathbb{N}}$ is nondecreasing. As $F$ is bounded from above on $S$, there exists $M>0$ such that $\sup_{n\in \mathbb{N}} F(\x_n) \leq M$. Then $\sup_{n\in \mathbb{N}} \theta_n \leq M$, and so $\theta_n \to \theta^*$ as $n \to +\infty$. Let $k \in \mathbb{N}$. Summing \eqref{eq:9} from $n=0$ to $k$, we have
\begin{equation*}
(\delta-2\overline{\nu})\sum_{n=0}^k \|\x_{n+1}-\x_n\|^2\leq \theta_{k+1}-\theta_0.
\end{equation*}
Letting $k \to +\infty$, we see that $\sum_{n=0}^{\infty} \|\x_{n+1}-\x_n\|^2<+\infty$. In particular,
$\|\x_{n+1}-\x_n\| \to 0$ as $n \to +\infty$. Thus, $F(\x_n) =\theta_n +\overline{\nu}\|\x_n-\x_{n-1}\|^2 \to \theta^*$ as $n \to +\infty$.

Next, to see the boundedness of $(\x_n)_{n \in \mathbb{N}}$, we observe from the nondecreasing property of $(\theta_n)_{n\in \mathbb{N}}$ that $F(\x_n) \geq \theta_n \geq \theta_0 = F(\x_0) -  \overline{\nu}\|\x_0-\x_{-1}\|^2 =F(x_0)$, 
where the last equality follows as $\x_{-1}=\x_0$. So, $(\x_n)_{n \in \mathbb{N}} \subseteq \{\x \in S
: F(\x) \geq F(\x_0)\}$, and hence $(\x_n)_{n \in \mathbb{N}}$ is bounded.

\ref{t:cvg_crit}:
Let $\overline{\x} =(\overline{x}_1,\dots,\overline{x}_m)$ be any cluster point of $(\x_n)_{n\in \mathbb{N}}$ and let $(\x_{k_n})_{n\in \mathbb{N}}$ be a subsequence of $(\x_n)_{n\in \mathbb{N}}$ such that $\x_{k_n}\to \overline{\x}$ as $n \to +\infty$. Then $\overline{\x}\in S$ and, by the asymptotic regularity, $\x_{k_n+1}\to \overline{\x}$ as $n \to +\infty$. Fix any $i\in \{1,\dots,m\}$. By the local Lipschitz continuity of $f_i$ and $g_i$, we have that $f_i(x_{i,k_n})\to f_i(\overline{x}_i)$, $g_i(x_{i,k_n})\to g_i(\overline{x}_i) >0$, and, by \cite[Corollary~1.81]{Mor06}, $(u_{i,k_n})_{n\in \mathbb{N}}$ and $(v_{i,k_n})_{n\in \mathbb{N}}$ are bounded. Noting also that
\begin{align*}
w_{i,k_n} &=\begin{cases}
\frac{g_i(x_{i,k_n})u_{i,k_n} -f_i(x_{i,k_n})v_{i,k_n}}{(g_i(x_{i,k_n}))^2} &\text{if~} f_i(x_{i,k_n}) >0,\\
0 & \text{if~} f_i(x_{i,k_n}) =0
\end{cases}\\ 
&\in \frac{g_i(x_{i,k_n}) \partial_L f_i(x_{i,k_n})-f_i(x_{i,k_n}) \partial_L g_i(x_{i,k_n}) }
{(g_i(x_{i,k_n}))^2},
\end{align*}
one sees $(w_{i,k_n})_{n\in \mathbb{N}}$ is bounded. Passing to a subsequence if necessary, we can assume that
\begin{equation*}
w_{i,k_n} \to \overline{w}_i \in \frac{g_i(\overline{x}_i) \partial_L f_i(\overline{x}_i)-f_i(\overline{x}_i) \partial_L g_i(\overline{x}_i) }
{g_i(\overline{x}_i)^2} \quad \text{as~~} n \to +\infty.
\end{equation*}
Now, replacing $n$ by $k_n$ in \eqref{eq:use1}, we have, for all $x_i \in S_i$ and all $n\in \mathbb{N}$, that
\begin{align}\label{eq:use2}
&h_{i,k_n+1}(x_i) -h_{i,k_n+1}(x_{i,k_n+1}) \\
&\qquad \leq -\tau_{k_n}\|x_{i,k_n+1}-x_{i,k_n}\|^2 +\tau_{k_n}\|x_i-x_{i,k_n}\|^2 +\langle w_{i,k_n},x_{i,k_n+1}-x_i \rangle \notag \\
&\qquad\qquad +2\tau_{k_n}\nu_{k_n}\langle x_{i,k_n+1}-x_i, x_{i,k_n}-x_{i,k_n-1}\rangle. \notag 
\end{align}
We shall split the proof in two following cases. 

\emph{Case~1:} $h$ is continuous on $S \cap \dom h$. Then $\lim_{n\to +\infty} h(\x_{k_n}) =h(\overline{\x})$, and so $\lim_{n\to +\infty} F(\x_n) =\lim_{n\to +\infty} F(\x_{k_n}) =F(\overline{\x})$.
Letting $n \to +\infty$ in \eqref{eq:use2}, we derive that, for all $x_i \in S_i$, 
\begin{equation*}
h(\overline{x}_1,\dots,\overline{x}_{i-1},x_i,\overline{x}_{i+1},\dots,\overline{x}_m) -h(\overline{\x}) \leq \overline{\tau}\|x_i-\overline{x}_i\|^2 +\langle \overline{w}_i, \overline{x}_i-x_i \rangle,
\end{equation*}
which means  
\begin{align*}
\overline{x}_i&\in \argmin_{x_i\in S_i} \{-h(\overline{x}_1,\dots,\overline{x}_{i-1},x_i,\overline{x}_{i+1},\dots,\overline{x}_m)
+\overline{\tau} \|x_i-\overline{x}_i\|^2 -\langle \overline{w}_i, x_i \rangle\} \\
&=\argmin_{x_i\in \mathcal{H}_i} \{(-h+\delta_S)(\overline{x}_1,\dots,\overline{x}_{i-1},x_i,\overline{x}_{i+1},\dots,\overline{x}_m)
+\overline{\tau} \|x_i-\overline{x}_i\|^2 -\langle \overline{w}_i, x_i \rangle\}.
\end{align*}
It follows that
\begin{equation*}
0\in \partial_L^{x_i} (-h+\delta_S)(\overline{\x}) -\overline{w}_i \subseteq \partial_L^{x_i} (-h+\delta_S)(\overline{\x}) +\frac{ -g_i(\overline{x}_i)\partial_L f_i(\overline{x}_i) +f_i(\overline{x}_i)\partial_L g_i(\overline{x}_i)}{g_i(\overline{x}_i)^2}. 
\end{equation*}
As this inclusion holds for any $i\in \{1,\dots,m\}$, $\overline{\x}$ is a lifted coordinate-wise stationary point for \eqref{eq:prob}.

\emph{Case~2:} $m =1$. Then \eqref{eq:use2} reduces to, for all $\x \in S$ and all $n\in \mathbb{N}$,
\begin{align*}
h({\x}) -h(\x_{k_n+1}) &\leq -\tau_{k_n}\|\x_{k_n+1}-\x_{k_n}\|^2 +\tau_{k_n}\|\x-\x_{k_n}\|^2 +\langle \bw_{k_n},\x_{k_n+1}-\x \rangle \\
&\hspace{4cm} +2\tau_{k_n}\nu_{k_n}\langle \x_{k_n+1}-\x, \x_{k_n}-\x_{k_n-1}\rangle,
\end{align*}
where $\bw_{k_n}\to \overline{\bw} \in \frac{ g_1(\overline{\x})\partial_L f_1(\overline{\x}) -f_1(\overline{\x})\partial_L g_1(\overline{\x})}{g_1(\overline{\x})^2}$ as $n\to +\infty$. Letting $\x=\overline{\x}$ and $n \to +\infty$, one has
$\liminf_{n \to +\infty} h(\x_{k_n+1}) \geq h(\overline{\x})$, which yields $\lim_{n \to +\infty} h(\x_{k_n+1}) =h(\overline{\x})$ due to the lower semicontinuity of $-h$. By arguing as in \emph{Case~1}, $\overline{\x}$ is a lifted coordinate-wise stationary point of \eqref{eq:prob}.

Finally, if Assumption \ref{a:add} holds, then Lemma~\ref{l:stationary} implies that  $\overline{x}$ is a stationary point for \eqref{eq:prob}. 
\end{proof}

We now comment on the assumptions imposed on the previous theorem. In particular, we see that they are quite general, and are all satisfied by our motivating examples.

\begin{remark}[comments on the assumptions]
\label{remark:parameter}
In addition to Assumption~\ref{a:standing}, we also assume in Theorem~\ref{t:cvg} that the objective function $F$ is bounded from above on the feasible set $S$ and that $\{\x \in S: F(\x) \geq F(\x_0)\}$ is bounded. These assumptions are trivially satisfied in the case when $S$ is a compact set (as in our three motivating examples). More generally, they are also satisfied in the case when $-F$ is a coercive function on the set $S$ (noting that we are considering a maximization formulation), which is a standard assumption in the optimization literature.

Finally, in order to obtain that every cluster point is a stationary point, we also assume that $\limsup_{n \to \infty} \tau_n=\overline{\tau}<+\infty$ and Assumption~\ref{a:add} holds. When $S$ is compact, the first assumption can be easily satisfied with $\tau_n =\delta +\max_{1\leq i\leq m} \{\frac{1}{2}(2y_{i,n}\alpha_i+y_{i,n}^2\beta_i)\}$ and $\overline{\tau} =\delta +\max_{1 \leq i \leq m} \{\frac{\sqrt{M_i}}{m_i} \alpha_i + \frac{M_i}{2m_i^2} \beta_i\}$, where $M_i:=\max_{x_i \in S_i} f_i(x_i)$ and $m_i:=\min_{x \in S_i}g_i(x_i)$. Also, it can be directly verified that Assumption~\ref{a:add} is satisfied by our three motivation examples in the introduction.  
\end{remark}

\begin{remark}[convergence to stronger stationary points]
A close inspection of the proof shows that one can obtain a stronger conclusion in Theorem~\ref{t:cvg} for the cluster point $\overline{\x} =(\overline{x}_1,\dots,\overline{x}_m)$. Indeed, the cluster point $\overline{\x}$ satisfies the following stronger stationarity notion: for each $i\in \{1,\dots,m\}$,
\begin{equation}\label{eq:stronger_stationary}
\overline{x}_i\in \argmin_{x_i\in S_i} \{-h(\overline{x}_1,\dots,\overline{x}_{i-1},x_i,\overline{x}_{i+1},\dots,\overline{x}_m)
+\overline{\tau} \|x_i-\overline{x}_i\|^2 -\langle \overline{w}_i, x_i \rangle\}
\end{equation}
for some $\overline{w}_i \in \frac{g_i(\overline{x}_i) \partial_L f_i(\overline{x}_i)-f_i(\overline{x}_i) \partial_L g_i(\overline{x}_i) }
{g_i(\overline{x}_i)^2}$. This relation implies that $\overline{\x}$ is a lifted coordinate-wise stationary point for \eqref{eq:prob}. Moreover, in the case when $m=1$, and $f_1, g_1$ are continuously differentiable, \eqref{eq:stronger_stationary} reduces to
\begin{equation*}
\overline{\x}\in \prox_{\frac{1}{2\overline{\tau}}(-h+\delta_S)}\left(\overline{\x}+\frac{1}{2\overline{\tau}}\nabla\left(\frac{f_1}{g_1}\right)(\overline{\x})\right),
\end{equation*}
which corresponds to the notion of a \emph{$L$-stationary point} with $L=2\overline{\tau}$ \cite[Definition~4.8]{BH18}, a stronger notion than the usual one of a stationary point.
\end{remark}

We now consider the global convergence of the full sequence generated by Algorithm~\ref{algo:epasa}. 
Recall that a proper lower semicontinuous function $f\colon \mathcal{H}\to \left(-\infty, +\infty\right]$ is said to satisfy the \emph{KL property} \cite{Kur98,Loj63} at $\overline{\x}\in \dom\partial_L f$ if there exist a neighborhood $U$ of $\overline{\x}$, $\eta \in \left(0,+\infty\right]$, and a continuous concave function $\varphi\colon \left[0, \eta\right)\to \mathbb{R}_+$ such that $\varphi(0) =0$, $\varphi$ is continuously differentiable with $\varphi' >0$ on $\left(0, \eta\right)$, and, for all $\x\in U$ with $f(\overline{\x}) <f(\x) <f(\overline{\x})+\eta$, 
\begin{equation*}
\varphi'(f(\x)-f(\overline{\x}))\dist(0,\partial_L f(\x))\geq 1.
\end{equation*}
If $f$ satisfies the KL property at any $\overline{\x}\in \dom\partial_L f$, then it is called a \emph{KL function}. We say that $f$ has the \emph{KL property at $\overline{\x}\in \dom\partial_L f$ with exponent $\alpha$} if it satisfies the KL property at $\overline{\x}\in \dom\partial_L f$ and the corresponding function $\varphi$ (often referred as desingularization function) can be chosen as $\varphi(s) =\gamma s^{1-\alpha}$ for some $\gamma\in \RPP$ and $\alpha\in \left[0,1\right)$. If $f$ is a KL function and has the same exponent $\alpha$ at any $\overline{\x}\in \dom\partial_L f$, then it is called a \emph{KL function with exponent $\alpha$}.

\begin{theorem}[global convergence]
\label{t:global}
Let $(\x_n)_{n\in \mathbb{N}}$ be the sequence generated by Algorithm~\ref{algo:epasa}. Suppose that Assumption~\ref{a:standing} holds, that $F$ is bounded from above on $S$, that the set $\{\x\in S: F(\x) \geq F(\x_0)\}$ is bounded, that, for each $i\in \{1,\dots,m\}$, $f_i$ and $g_i$ are continuously differentiable on an open set containing $S_i$, and that $G(\x,\bu) :=-F(\x) +\delta_S(\x) +\overline{\nu}\|\x-\bu\|^2$ satisfies the KL property at $(\overline{\x},\overline{\x})$ for all $\overline{\x} \in \dom\partial_L (-F+\delta_S)$. Suppose further that $\limsup_{n\to +\infty} \tau_n =\overline{\tau} <+\infty$, that either $m =1$ or $h$ is a differentiable function on an open set containing $S$ whose gradient is Lipschitz continuous on $S$, and that there exist $\varepsilon, \ell\in \RPP$ satisfying
\begin{multline*}
\text{for all~} i\in \{1,\dots,m\},\ \text{for all~} x,x'\in S_i,\\ \|x-x'\|\leq \varepsilon \implies \left\|\nabla\left(\frac{f_i}{g_i}\right)(x)-\nabla\left(\frac{f_i}{g_i}\right)(x')\right\|\leq \ell\|x-x'\|.
\end{multline*}
Then $\displaystyle \sum_{n=0}^{+\infty}\|\x_{n+1}-\x_n\|<+\infty$ and the sequence $(\x_n)_{n\in \mathbb{N}}$ converges to a stationary point $\x^*$ for \eqref{eq:prob}.\\ Moreover, if $G$ satisfies the KL property with exponent $\alpha\in [0,1)$ at $(\overline{\x},\overline{\x})$ for all $\overline{\x} \in \dom\partial_L(-F+\delta_S)$, then exactly one of the following alternatives holds:
\begin{enumerate}
\item
\emph{(Finite convergence)}
$\alpha =0$ and there exists $n_0\in \mathbb{N}$ such that, for all $n\geq n_0$, $\x_n =\x^*$.
\item
\emph{(Linear convergence)}
$\alpha\in (0,\frac{1}{2}]$ and there exist $\gamma\in \RPP$ and $\rho\in \left(0,1\right)$ such that, for all $n\in \mathbb{N}$, $\|\x_n-\x^*\|\leq \gamma\rho^{\frac{n}{2}}$ and $|F(\x_n) -F(\x^*)|\leq \gamma\rho^n$.
\item
\emph{(Sublinear convergence)}
$\alpha\in (\frac{1}{2},1)$ and there exists $\gamma\in \RPP$ such that, for all $n\in \mathbb{N}$, $\|\x_n-\x^*\|\leq \gamma n^{-\frac{1-\alpha}{2\alpha-1}}$ and $|F(\x_n) -F(\x^*)|\leq \gamma n^{-\frac{2-2\alpha}{2\alpha-1}}$.
\end{enumerate}
\end{theorem}
\begin{proof}
Let $\z_n :=(\x_{n+1},\x_n)$ for $n\in \mathbb{N}$ and $\Omega$ be the set of cluster points of $(\z_n)_{n\in \mathbb{N}}$.
We derive from Theorem~\ref{t:cvg} that the sequence $(\z_n)_{n\in \mathbb{N}}$ in $S \times S$ is bounded, that, for all $n\in \mathbb{N}$,
\begin{equation}\label{eq:decrease'}
G(\z_{n+1})+\alpha \|\x_{n+2}-\x_{n+1}\|^2\leq G(\z_n) \quad\text{with~} \alpha :=\delta -2\overline{\nu} >0,
\end{equation}
and that, for all $\overline{\z}\in \Omega$, one has $\overline{\z} =(\overline{\x},\overline{\x})$ with  $\overline{\x}\in S$ being a stationary point for \eqref{eq:prob} and $F(\x_n) \to F(\overline{\x})$ as $n\to +\infty$.
Therefore, $\overline{\x} \in \dom\partial_L(-F + \delta_S)$ and $G(\z_n) =G(\x_{n+1},\x_n) =-F(\x_{n+1}) +\overline{\nu}\|\x_{n+1}-\x_n\|^2 \to -F(\overline{\x})$ as $n \to +\infty$.

Now, for all $n\in \mathbb{N}$, since $\partial_L G(\z_n) =(\partial_L(-F+\delta_S)(\x_{n+1}) +2\overline{\nu}(\x_{n+1}-\x_n), 2\overline{\nu}(\x_{n}-\x_{n+1}))^\top$,
\begin{align}\label{eq:subdiffG}
&\dist(0,\partial_L G(\z_n)) \\
&\quad =\sqrt{\dist\left(0,\partial_L(-F+\delta_S)(\x_{n+1}) +2\overline{\nu}(\x_{n+1}-\x_n)\right)^2 +(2\overline{\nu})^2\|\x_{n+1}-\x_n\|^2} \notag \\
&\quad \leq \dist\left(0,\partial_L(-F+\delta_S)(\x_{n+1}) +2\overline{\nu}(\x_{n+1}-\x_n)\right) + 2\overline{\nu}\|\x_{n+1}-\x_n\| \notag \\
&\quad \leq \dist(0,\partial_L(-F+\delta_S)(\x_{n+1})) + 4\overline{\nu}\|\x_{n+1}-\x_n\|. \notag
\end{align}
We shall estimate $\dist(0,\partial_L(-F+\delta_S)(\x_{n+1}))$. From Step~\ref{step:main} of Algorithm~\ref{algo:epasa} and noting that $f_i,g_i$ are continuously differentiable on an open set that contains $S_i$, we have, for all $i\in \{1,\dots,m\}$ and all $n\in \mathbb{N}$, that
\begin{align}\label{eqLoop}
0 &\in \partial_L (-h_{i,n+1}+\delta_{S_i})(x_{i,n+1}) +2\tau_n(x_{i,n+1}-z_{i,n}) -w_{i,n} \notag \\
&= \partial_L \left(-h_{i,n+1}-\frac{f_i}{g_i}+\delta_{S_i}\right)(x_{i,n+1}) +2\tau_n(x_{i,n+1}-z_{i,n}) +(w_{i,n+1}-w_{i,n}),
\end{align}
where $w_{i,n} =\frac{ g_i(x_{i,n})\nabla f_i(x_{i,n}) -f_i(x_{i,n})\nabla g_i(x_{i,n})}{(g_i(x_{i,n}))^2} =\nabla\left(\frac{f_i}{g_i}\right)(x_{i,n})$.
Since $\limsup_{n\to +\infty} \tau_n =\overline{\tau} <+\infty$ and, for each $i\in \{1,\dots,m\}$, $\lim_{n\to +\infty} \|x_{i,n+1}-x_{i,n}\| =0$ (by Theorem~\ref{t:cvg}\ref{t:cvg_seq}), there exists $n_0\geq 0$ such that, for all $n\geq n_0$,
\begin{equation*}
\tau_n\leq 2\overline{\tau} \text{~~and~~} \|x_{i,n+1}-x_{i,n}\|\leq \varepsilon.
\end{equation*}
Then, for all $i\in \{1,\dots,m\}$ and all $n\geq n_0$, we derive from $\nu_n\leq \overline{\nu}/\tau_n$ that
\begin{align}\label{eq:xizi}
\|2\tau_n(x_{i,n+1}-z_{i,n})\| &= \|2\tau_n(x_{i,n+1}-x_{i,n}) -2\tau_n\nu_n(x_{i,n}-x_{i,n-1})\| \notag \\
&\leq 4\overline{\tau}\|x_{i,n+1}-x_{i,n}\| +2\overline{\nu}\|x_{i,n}-x_{i,n-1}\|
\end{align}
and from the assumption on $\nabla(\frac{f_i}{g_i})$ that
\begin{equation}\label{eq:wiwi}
\|w_{i,n+1}-w_{i,n}\| =\left\|\nabla\left(\frac{f_i}{g_i}\right)(x_{i,n+1})-\nabla\left(\frac{f_i}{g_i}\right)(x_{i,n})\right\|\leq \ell\|x_{i,n+1}-x_{i,n}\|.
\end{equation}
We split the discussion into the following cases. 

\emph{Case~1:} $h$ is a differentiable function on an open set containing $S$ whose gradient is Lipschitz continuous on $S$ with modulus $\ell_h$. Then, it follows from \eqref{eqLoop} that, for all $i\in \{1,\dots,m\}$ and all $n\in \mathbb{N}$,
\begin{equation*}
0 \in -\nabla h_{i,n+1}(x_{i,n+1}) +\partial_L \left(-\frac{f_i}{g_i}+\delta_{S_i}\right)(x_{i,n+1})  +2\tau_n(x_{i,n+1}-z_{i,n}) +(w_{i,n+1}-w_{i,n}),
\end{equation*}
which yields
\begin{multline*}
-\left[\nabla_{x_i} h(\x_{n+1}) -\nabla h_{i,n+1}(x_{i,n+1})\right] -2\tau_n(x_{i,n+1}-z_{i,n}) -(w_{i,n+1}-w_{i,n}) \\ \in \partial_L^{x_i}(-F+\delta_S)(\x_{n+1}).
\end{multline*}
Combining with \eqref{eq:xizi} and \eqref{eq:wiwi}, we deduce that, for all $n\geq n_0$,
\begin{align*}
&\dist(0,\partial_L(-F+\delta_S)(\x_{n+1})) \\
&\quad \leq \sum_{i=1}^m \dist(0,\partial_L^{x_i}(-F+\delta_S)(\x_{n+1})) \\
&\quad \leq \sum_{i=1}^m \|\nabla_{x_i} h(\x_{n+1}) -\nabla h_{i,n+1}(x_{i,n+1})\| +(4\overline{\tau}+\ell)\sum_{i=1}^m \|x_{i,n+1}-x_{i,n}\| \\ 
&\quad\quad +2\overline{\nu}\sum_{i=1}^m \|x_{i,n}-x_{i,n-1}\| \\
&\quad \leq m\ell_h\|\x_{n+1}-\x_n\| +(4\overline{\tau}+\ell)\sqrt{m}\|\x_{n+1}-\x_n\| +2\overline{\nu}\sqrt{m}\|\x_n-\x_{n-1}\|,
\end{align*}
where the last inequality holds due to the assumption that $\nabla h$ is Lipschitz continuous with modulus $\ell_h$ on $S$. By using \eqref{eq:subdiffG}, there exists $K\in \RPP$ such that, for all $n\geq n_0$,
\begin{equation*}
\dist(0,\partial_L G(\z_n)) \leq K \left(\|\x_{n+1}-\x_n\| +\|\x_n-\x_{n-1}\|\right).
\end{equation*}
The first conclusion then follows by applying \cite[Theorem~5.1]{BDL20} (with $I =\{1,2\}$, $\lambda_1 =\lambda_2 =1/2$, $\Delta_n =2K\|\x_{n+2}-\x_{n+1}\|$, $\alpha_n \equiv \frac{\alpha}{4K^2}>0$, $\beta_n \equiv 1$, and $\varepsilon_n \equiv 0$).

\emph{Case~2:} $m =1$. In this case, we derive from \eqref{eqLoop} that, for all $n\in \mathbb{N}$,
\begin{equation*}
-2\tau_n(\x_{n+1}-\z_n) -(\bw_{n+1}-\bw_n)\in \partial_L(-F+\delta_S)(\x_{n+1}).
\end{equation*}
Thus, \eqref{eq:xizi} and \eqref{eq:wiwi} imply that, for all $n\geq n_0$,
\begin{align*}
\dist(0,\partial_L(-F+\delta_S)(\x_{n+1})) &\leq \|2\tau_n(\x_{n+1}-\z_n)\| +\|\bw_{n+1}-\bw_n\| \\
&\leq (4\overline{\tau}+\ell)\|\x_{n+1}-\x_n\| +2\overline{\nu}\|\x_n-\x_{n-1}\|.
\end{align*}
Proceeding as in \emph{Case~1}, we also obtain the first conclusion.

As the remaining conclusions are rather standard, we omit the proof here and refer the readers to \cite{AB09,BDL20,Laszlo21,LP18}.
\end{proof}

As stated in the preceding theorem, the KL exponent of the merit function for the model problem completely determines the convergence rate of the proposed algorithm. On the other hand, finding or estimating the KL exponent of a nonsmooth and nonconvex function is, in general, highly challenging. Some recent progresses in identifying KL exponents for non-fractional problems can be found in \cite{LP18,YLP22}.  In the next section, we will derive KL exponents of the corresponding merit functions for various classes of structured fractional programming problems.

\section{KL exponents for structured fractional programs}
\label{s:KL}
In this section, we derive the KL exponent of the associated merit functions of three classes of structured fractional programs:
sum-of-ratios fractional quadratic programs with spherical constraints, generalized eigenvalue problems with cardinality regularization and generalized eigenvalue problems with sparsity constraints. In particular, we establish that, for the last two classes of fractional programs, the KL exponent is $1/2$. As a consequence,  the proposed Algorithm~\ref{algo:epasa} exhibits linear convergence for these two classes of fractional programs.

We first see that the KL exponent for the merit function associated with \eqref{eq:prob} can be computed by a merit function associated with the equivalent problem~\eqref{eq:prob1}. To do this, we need the following result from \cite{LP18}.

\begin{lemma}[cf. {\cite[Theorem 3.6]{LP18}}]
\label{l:LP}
Let $f$ be a proper lower semicontinuous function. Suppose that $f$ satisfies the KL property at $\overline{\x} \in \dom\partial_L f$ with exponent $\alpha \in [\frac{1}{2},1)$. Then, for all $\rho \geq 0$, $\widetilde{f}(\x,\bu)=f(\x)+\rho \|\x-\bu\|^2$ satisfies  the KL property with exponent $\alpha$ at $(\overline{\x},\overline{\x})$.
\end{lemma}

\begin{proposition}
\label{p:bridge}
Suppose that Assumption~\ref{a:standing} holds and that, for each $i\in \{1,\dots,m\}$, $f_i$ and $g_i$ are continuously differentiable on $S_i$.
Let $P(\x,\y) =-h(\x)- H(\x,\y)+\delta_S(\x)$, where $H(\x,\y) = \sum_{i=1}^m \left[2y_i\sqrt{f_i(x_i)}-y_i^2g_i(x_i)\right]$. Let $\overline{\x} =(\overline{x}_1,\dots,\overline{x}_m)\in S$ and  $\overline{\y} =(\overline{y}_1,\dots,\overline{y}_m)\in \mathbb{R}^m$ with $\overline{y_i} =\frac{\sqrt{f_i(\overline{x}_i)}}{g_i(\overline{x}_i)}$. Suppose further that $h$ is continuous around $\overline{\x}$ and that $P$ satisfies the KL property with exponent $\alpha \in [0,1)$ at $(\overline{\x},\overline{\y}) \in S\times \mathbb{R}^m$. Then
\begin{equation*}
\Phi(\x) :=-h(\x)-\sum_{i=1}^m\frac{f_i(x_i)}{g_i(x_i)} +\delta_S(\x)
\end{equation*}
satisfies the KL property with exponent $\alpha$ at $\overline{\x}$. In particular, for all $\rho\geq 0$,
\begin{equation*}
G(\x,\bu):=-h(\x)-\sum_{i=1}^m\frac{f_i(x_i)}{g_i(x_i)}+\rho \|\x-\bu\|^2 +\delta_S(\x)
\end{equation*}
satisfies the KL property with exponent $\alpha' =\max\{\alpha,\frac{1}{2}\}$ at $(\overline{\x},\overline{\x})$.
\end{proposition}
\begin{proof}
As $P$ satisfies the KL property with exponent $\alpha \in [0,1)$ at $(\overline{\x},\overline{\y}) \in \mathcal{H} \times \mathbb{R}^m$, there exist $\delta,\eta,c>0$ such that, for all $(\x,\y)$ with  $\|(\x,\y)-(\overline{\x},\overline{\y})\| \leq \delta$ and $P(\overline{\x},\overline{\y})<P(\x,\y)<P(\overline{\x},\overline{\y})+\eta$, one has $\dist(0, \partial_L P(\x,\y)) \geq c\left[P(\x,\y)-P(\overline{\x},\overline{\y}) \right]^{\alpha}$.
It follows that, for all $(\x,\y)$ with  $\|(\x,\y)-(\overline{\x},\overline{\y})\| \leq \delta$ and $P(\x,\y)<P(\overline{\x},\overline{\y})+\eta$,
\begin{equation}\label{eq:KLexp}
\left[\dist(0, \partial_L P(\x,\y))\right]^{\frac{1}{\alpha}} \geq c^{\frac{1}{\alpha}}\left[P(\x,\y)-P(\overline{\x},\overline{\y}) \right].
\end{equation}
Here, we drop the condition  $P(\overline{\x},\overline{\y})<P(\x,\y)$ because \eqref{eq:KLexp} trivially holds otherwise.
For each $\x$, let $\y_\x =(y_{1,\x},\dots,y_{m,\x})$ with $y_{i,\x} =\frac{\sqrt{f_i({x_i})}}{g_i({x_i})}$ for $i\in \{1,\dots,m\}$. Then $\y_{\overline{\x}} =\overline{\y}$. Moreover, by the continuity of $H$ and $h$, there exists $\delta_1 \in (0,\delta)$ such that for all $\x \in S$ with $\|\x-\overline{\x}\| \leq \delta_1$ one has
$\|(\x,\y_{\x})-(\overline{\x},\overline{\y})\| \leq \delta$ and $P(\x,\y_{\x})<P(\overline{\x},\overline{\y})+\eta$. Therefore, from \eqref{eq:KLexp} we derive that, for all $\x \in S$ with $\|\x-\overline{\x}\| \leq \delta_1$,
\begin{align*}
\left[\dist(0, \partial_L P(\x,\y_\x)) \right]^{\frac{1}{\alpha}}
&\geq c^{\frac{1}{\alpha}}\left[-h(\x)-H(\x,\y_{\x}) +h(\overline{\x})+H(\overline{\x},\overline{\y})\right] \\
&= c^{\frac{1}{\alpha}}\left[-h(\x)-\sum_{i=1}^m\frac{f_i(x_i)}{g_i(x_i)} +h(\overline{\x})+\sum_{i=1}^m\frac{f_i(\overline{x}_i)}{g_i(\overline{x}_i)} \right].
\end{align*}
Now, we notice that $\partial_L P(\x,\y_\x) =\left(\partial_L (-h+\delta_S)(\x)+\partial_L^\x (-H)(\x,\y_\x), \partial_L^\y (-H)(\x,\y_\x)\right)$ and that, for each $i\in \{1,\dots,m\}$,
\begin{align*}
\partial_L^{x_i} (-H_i)(x_i,y_{i,\x}) &=\frac{-g_i(x_i)\nabla f_i(x_i) +f_i(x_i)\nabla g_i(x_i)}{[g_i(x_i)]^2} =\nabla\left(\frac{f_i}{g_i}\right)(x_i) \\ 
\text{and~} \partial_L^{y_i} (-H_i)(x_i,y_{i,\x}) &=0.
\end{align*}
Therefore, $\partial_L P(\x,\y_\x) =(\partial_L \Phi(\x), 0)$, and from here we deduce that, for all $\x \in S$ with $\|\x-\overline{\x}\| \leq \delta_1$,
\begin{equation*}
\left[\dist(0,\partial_L \Phi(\x))\right]^{\frac{1}{\alpha}} \geq c^{\frac{1}{\alpha}}\left[\Phi(\x)-\Phi(\overline{\x})\right] .
\end{equation*}
So, $\Phi$ satisfies the KL property with exponent $\alpha$ at $\overline{\x}$, and hence, also with exponent $\max\{\alpha,\frac{1}{2}\}$. By using Lemma~\ref{l:LP}, $G$ satisfies the KL property with exponent $\max\{\alpha,\frac{1}{2}\}$ at $(\overline{\x},\overline{\x})$.
\end{proof}

\subsection{Sum-of-ratios fractional quadratic programs with spherical constraint}

We now consider the following sum-of-ratios fractional quadratic program
\begin{equation}
\max_{\x =(x_1,\dots,x_m) \in \mathbb{R}^d} \x^\top A_0\x+ \ba_0^\top \x+\sum_{i=1}^m \frac{x_i^\top A_ix_i}{x_i^\top B_ix_i}
\quad\text{s.t.}\quad \|x_i\|=1, \ i \in \{1,\dots,m\},
\tag{FQP}
\end{equation}
where $A_0$ is a symmetric matrix and, for each $i\in \{1, \dots, m\}$, $A_i$ and $B_i$ are positive definite matrices. In the special cases of $m=1$ and  $\ba_0=0$, this reduces to the problem of maximizing the sum of a quadratic function and the Rayleigh quotient over the unit
sphere (motivating Example~\ref{ex:Rayleigh}). For this sum-of-ratios fractional  quadratic program, the corresponding merit function for the proposed Algorithm~\ref{algo:epasa} takes the form
\begin{equation*}
\widehat{\Phi}_{FQP}(\x,\bu) =-\left[\x^\top A_0\x +\ba_0^\top \x\right] -\sum_{i=1}^m \frac{x_i^\top A_ix_i}{x_i^\top B_ix_i} +\delta_{\Lambda_1\times \dots \times \Lambda_m}(\x) +\rho\|\x-\bu\|^2,
\end{equation*}
where $\Lambda_i =\{x_i\in \mathbb{R}^{d_i}: \|x_i\|=1\}$, $i\in \{1,\dots,m\}$, and $\rho \geq 0$. We shall investigate the KL exponent of this merit function. To this end, we use a fundamental result which provides an exponent estimate in the classical \L ojasiewicz gradient inequality for polynomials. 

\begin{lemma}[\L ojasiewicz gradient inequality {\cite[Theorem~4.2]{AK05}}]
\label{l:LGI}
Let $f$ be a polynomial on $\mathbb{R}^d$ with degree $p\in\mathbb{N}$. Suppose that $f(\overline{\x})=0$. Then there exist constants $\varepsilon, c>0$ such that, for all $\x\in\mathbb{R}^d$ with $\|\x-\overline{\x}\|\leq \varepsilon$, we have
\begin{equation}\label{loj}
\|\nabla f(\x)\|\geq c|f(\x)|^{1-\tau},
\text{~~where~~} \tau =\mathcal{R}(d,p)^{-1} \text{~~and~~}
\mathcal{R}(d,p) :=\begin{cases}
1 &\text{if~~} p =1, \\
p(3p-3)^{d-1} &\text{if~~} p\geq 2.
\end{cases}
\end{equation}
\end{lemma}

\begin{theorem}
Let $\Lambda =\Lambda_1\times \dots \times \Lambda_m$, where $\Lambda_i =\{x_i \in \mathbb{R}^{d_i}: \|x_i\|=1\}$, $i\in \{1,\dots,m\}$, with $\sum_{i=1}^m d_i=d$. Consider
\begin{equation*}
\Phi(\x) =-\left[\x^\top A_0\x+\ba_0^\top \x\right] -\sum_{i=1}^m \frac{x_i^\top A_ix_i}{x_i^\top B_ix_i} +\delta_{\Lambda}(\x),
\end{equation*}
where $A_0$ is a symmetric matrix and, for each $i\in \{1, \dots, m\}$, $A_i$ and $B_i$ are positive definite matrices. Then $\Phi$ satisfies the KL property with exponent $1-\tau$, where $\tau = \left[\mathcal{R}(d+3m+md,4)\right]^{-1}$. In particular, for all $\rho \geq 0$,
\begin{equation*}
\widehat{\Phi}_{FQP}(\x,\bu) =-\left[\x^\top A_0\x+\ba_0^\top \x\right] -\sum_{i=1}^m \frac{\x^\top A_i\x}{\x^\top B_i\x} +\delta_{\Lambda}(\x) +\rho \|\x-\bu\|^2
\end{equation*}
satisfies the KL property with exponent $1-\tau$ at $(\overline{\x},\overline{\x})$ for all $\overline{\x} \in \dom\partial_L \Phi$.
\end{theorem}
\begin{proof}
From Proposition~\ref{p:bridge} with $S =\Lambda$, $h(\x) =\x^\top A_0\x+\ba_0^\top \x$, $f_i(x_i) =x_i^\top A_ix_i$, and $g_i(x_i) =x_i^\top B_ix_i$, $i \in \{1, \dots, m\}$, it suffices to show that
\begin{equation*}
P(\x,\y) =-\left[\x^\top A_0\x+\ba_0^\top \x\right] -\sum_{i=1}^m \left[y_i\sqrt{x_i^\top A_ix_i}-y_i^2x_i^\top B_ix_i\right] +\delta_{\Lambda}(\x)
\end{equation*}
satisfies the KL property with exponent $1-\tau$ at $(\overline{\x},\overline{\y}) \in \Lambda \times \mathbb{R}^m$. To do this, let $(\overline{\x},\overline{\y}) \in \Lambda\times \mathbb{R}^m$ and let $\delta,\eta>0$ be such that, for all
$\|(\x,\y)-(\overline{\x},\overline{\y})\| \leq \delta$, one has $P(\overline{\x},\overline{\y}) <P(\x,\y) <P(\overline{\x},\overline{\y})+\eta$. Let $L_i =A_i^{1/2}$ for $i \in \{1, \dots, m\}$. We can write $P$ as
\begin{equation*}
P(\x,\y) =-\left[\x^\top A_0\x+\ba_0^\top \x\right] -\sum_{i=1}^m \left[y_i\|L_ix_i\|-y_i^2x_i^\top B_ix_i\right] +\sum_{i=1}^m \delta_{\Lambda_i}(x_i).
\end{equation*}
As $A_i$ is positive definite, we have $L_ix_i \neq 0$ for all $x_i \in \Lambda_i$ and all $i \in \{1, \dots, m\}$.
Let $f_0(\x) =\x^\top A_0\x+\ba_0^\top \x$. Then, for all $i \in \{1, \dots, m\}$,
\begin{align*}
\partial_L^{x_i} P(\x,\y) &=\left\{-\nabla_{x_i} f_0(\x) -y_i\frac{A_ix_i}{\sqrt{x_i^\top A_ix_i}} +2y_i^2B_ix_i +t_ix_i: t_i\in \mathbb{R}\right\} \\
&\hspace{5.5cm}\text{(Using Lemma~\ref{l:subdiff}\ref{l:subdiff_sphere})} \\
\text{and~} \partial_L^{y_i} P(\x,\y) &=-\|L_ix_i\| +2y_i x_i^\top B_ix_i,
\end{align*}
which imply that
\begin{multline*}
\dist(0,\partial_L P(\x,\y))^2\! =\!\sum_{i=1}^m \inf_{t_i\in \mathbb{R}} \!\left\{\left\|-\nabla_{x_i} f_0(\x) -y_i\frac{A_ix_i}{\sqrt{x_i^\top A_ix_i}} +2y_i^2B_ix_i +t_ix_i\right\|^2\!\right\} \\ +\sum_{i=1}^m (-\|L_ix_i\|+2y_i x_i^\top B_ix_i)^2.
\end{multline*}
For all $(\x,\y)\in \Lambda \times \mathbb{R}^m$ and all $i\in \{1,\dots,m\}$, one has $\|x_i\|=1$, and so,
\begin{align*}
&\left\|-\nabla_{x_i} f_0(\x) -y_i\frac{A_ix_i}{\sqrt{x_i^\top A_ix_i}} +2y_i^2B_ix_i +t_ix_i\right\|^2 \\
&\quad =\left\|-\nabla_{x_i} f_0(\x) -y_i\frac{A_ix_i}{\sqrt{x_i^\top A_ix_i}} +2y_i^2B_ix_i\right\|^2 \\
&\qquad +2t_ix_i^\top \Big(-\nabla_{x_i} f_0(\x) -y_i\frac{A_ix_i}{\sqrt{x_i^\top A_ix_i}} +2y_i^2B_ix_i\Big) +t_i^2,
\end{align*}
from which we have
\begin{multline}\label{eq:olala}
\dist(0,\partial_L P(\x,\y))^2 =\sum_{i=1}^m \left\|-\nabla_{x_i} f_0(\x) -y_i\frac{A_ix_i}{\sqrt{x_i^\top A_ix_i}} +2y_i^2B_ix_i +t_{x_i,y_i}x_i\right\|^2 \\ +\sum_{i=1}^m (-\|L_ix_i\|+2y_i x_i^\top B_ix_i)^2,
\end{multline}
where $t_{x_i,y_i} := x_i^\top \big(\nabla_{x_i} f_0(\x)\big) +y_i\sqrt{x_i^\top A_ix_i} -2y_i^2x_i^\top B_ix_i$.

Now, let us consider $f\colon \mathbb{R}^d \times \mathbb{R}^{m} \times \mathbb{R}^{md} \times \mathbb{R}^m \times \mathbb{R}^m\to \mathbb{R}$ defined by
\begin{multline*}
f(\x,\y,\bu,\blambda,\bmu) =-f_0(\x) -\sum_{i=1}^m \left[y_i(L_ix_i)^\top u_i -y_i^2x_i^\top B_ix_i\right] \\ +\sum_{i=1}^m\lambda_i(\|u_i\|^2-1)+\sum_{i=1}^m\mu_i(\|x_i\|^2-1),
\end{multline*}
and let $\widehat{f} =f -r$, where $r =f(\overline{\x},\overline{\y},\overline{\bu},\overline{\blambda},\overline{\bmu})$ with $\overline{u}_i =\frac{L_i\overline{x}_i}{\|L_i\overline{x}_i\|}$, $\overline{\lambda}_i =\frac{\overline{y}_i\|L_i\overline{x}_i\|}{2}$ and $\overline{\mu}_i =\frac{t_{\overline{x}_i,\overline{y}_i}}{2}$ for all $i \in \{1, \dots, m\}$.
Clearly, $\widehat{f}$ is a polynomial on $\mathbb{R}^{d+3m+md}$ of degree $4$. By 
Lemma~\ref{l:LGI}, there exist $\delta_0 >0$ and $c>0$ such that, for all $(\x,\y,\bu,\blambda,\bmu)$ with $\|(\x,\y,\bu,\blambda,\bmu)-(\overline{\x},\overline{\y},\overline{\bu},\overline{\blambda},\overline{\bmu})\| \leq \delta_0$,
\begin{align*}
\|\nabla f(\x,\y,\bu,\blambda,\bmu)\| &=\|\nabla \widehat{f}(\x,\y,\bu,\blambda,\bmu)\|\geq c|\widehat{f}(\x,\y,\bu,\blambda,\bmu)|^{1-\tau} \\
&=c|f(\x,\y,\bu,\blambda,\bmu)-f(\overline{\x},\overline{\y},\overline{\bu},\overline{\blambda},\overline{\bmu})|^{1-\tau},
\end{align*}
where $\tau = \left[\mathcal{R}(d+3m+md,4)\right]^{-1}$.
Let $\bu_{\x} =(u_{1,\x},\dots,u_{m,\x})$, $\blambda_{\x,\y} =(\lambda_{1,\x,\y},\dots,\lambda_{m,\x,\y})$, and $\bmu_{\x,\y} =(\mu_{1,\x,\y},\dots,\mu_{m,\x,\y})$ with, for all $i \in \{1, \dots, m\}$, 
\begin{equation*}
u_{i,\x} =\frac{L_ix_i}{\|L_ix_i\|},\quad
\lambda_{i,\x,\y} =\frac{y_i\|L_ix_i\|}{2},\quad\text{and~}
\mu_{i,\x,\y} =\frac{t_{x_i,y_i}}{2}.
\end{equation*}
Shrinking $\delta>0$ if necessary, we can assume that, for all $(\x,\y) \in \Lambda \times \mathbb{R}^m$ with $\|(\x,\y)-(\overline{\x},\overline{\y})\| \leq \delta$,
\begin{equation*}
\|(\x,\y,\bu_{\x},\blambda_{\x,\y},\bmu_{\x,\y})-(\overline{\x},\overline{\y},\overline{\bu},\overline{\blambda},\overline{\bmu})\| \leq \delta_0,
\end{equation*}
which implies $\|\nabla f(\x,\y,\bu_{\x},\blambda_{\x,\y},\bmu_{\x,\y})\|^2 \geq c^2|f(\x,\y,\bu_{\x},\blambda_{\x,\y},\bmu_{\x,\y})-f(\overline{\x},\overline{\y},\overline{\bu},\overline{\blambda},\overline{\bmu})|^{2(1-\tau)}$.
Note that, for all $i \in \{1, \dots, m\}$,
\begin{equation*}
\begin{cases}
\nabla_{x_i} f(\x,\y,\bu,\blambda,\bmu) &=-\nabla_{x_i} f_0(\x) -(y_i L_i^\top u_i-2y_i^2B_ix_i) +2\mu_ix_i, \\
\nabla_{y_i} f(\x,\y,\bu,\blambda,\bmu) &=-(L_ix_i)^\top u_i +2y_ix_i^\top B_ix_i, \\
\nabla_{u_i} f(\x,\y,\bu,\blambda,\bmu) &=-y_iL_ix_i +2\lambda_i u_i, \\
\nabla_{\lambda_i} f(\x,\y,\bu,\blambda,\bmu) &=\|u_i\|^2-1, \\
\nabla_{\mu_i} f(\x,\y,\bu,\blambda,\bmu) &=\|x_i\|^2-1.
\end{cases}
\end{equation*}
Direct verification shows that,  for all $(\x,\y) \in \Lambda \times \mathbb{R}^m$ with $\|(\x,\y)-(\overline{\x},\overline{\y})\| \leq \delta$ and all $i \in \{1, \dots, m\}$, one has
\begin{equation*}
\begin{cases}
\nabla_{x_i} f(\x,\y,\bu_{\x},\blambda_{\x,\y},\bmu_{\x,\y})&= -\nabla_{x_i} f(\x) -\left (y_i\frac{A_ix_i}{\sqrt{x_i^\top A_ix_i}}-2y_i^2B_i\x_i \right) +t_{x_i,y_i}x_i, \\
\nabla_{y_i} f(\x,\y,\bu_{\x},\blambda_{\x,\y},\bmu_{\x,\y}) &=-\|L_ix_i\| +2y_ix_i^\top B_ix_i, \\
\nabla_{u_i} f(\x,\y,\bu_{\x},\blambda_{\x,\y},\bmu_{\x,\y}) &=0, \\
\nabla_{\lambda_i} f(\x,\y,\bu_{\x},\blambda_{\x,\y},\bmu_{\x,\y}) &=0, \\
\nabla_{\mu_i} f(\x,\y,\bu_{\x},\blambda_{\x,\y},\bmu_{\x,\y}) &=0
\end{cases}
\end{equation*}
and also
\begin{equation*}
f(\x,\y,\bu_{\x},\blambda_{\x,\y},\bmu_{\x,\y}) =P(\x,\y) \text{~~and~~}
f(\overline{\x},\overline{\y},\overline{\bu},\overline{\blambda},\overline{\bmu}) =P(\overline{\x},\overline{\y}).
\end{equation*}
These together with \eqref{eq:olala} imply that, for all $(\x,\y) \in \Lambda \times \mathbb{R}^m$ with $\|(\x,\y)-(\overline{\x},\overline{\y})\| \leq \delta$ and $P(\overline{\x},\overline{\y})<P(\x,\y)<P(\overline{\x},\overline{\y})+\eta$,
\begin{equation*}
\dist(0,\partial_L P(\x,\y)) \geq c\left|P(\x,\y)-P(\overline{\x},\overline{\y})\right|^{1-\tau} =c\left[P(\x,\y)-P(\overline{\x},\overline{\y})\right]^{1-\tau}.
\end{equation*}
Thus, $P$ satisfies the KL property with exponent $1-\tau$, and the conclusion follows.
\end{proof}

\subsection{Generalized eigenvalue problem with cardinality regularization}

Consider the generalized eigenvalue problem with cardinality regularization
\begin{equation}
\max_{\x \in \mathbb{R}^d} \frac{\x^\top A_1\x}{\x^\top B_1\x}- \lambda \|\x\|_0 
\quad\text{s.t.}\quad \|\x\|=1,
\tag{GEP}
\end{equation}
where $A_1,B_1$ are symmetric matrices such that $A_1$ is positive semidefinite and $B_1$ is positive definite, and $\lambda>0.$ For this generalized eigenvalue problem with cardinality regularization the corresponding merit function for the proposed Algorithm~\ref{algo:epasa} takes the form
\begin{equation*}
\widehat{\Phi}_{GEP}(\x,\bu) =\frac{\x^\top A \x}{\x^\top B \x} +\lambda \|\x\|_0 +\delta_{\Lambda}(\x) +\rho \|\x-\bu\|^2,
\end{equation*}
with $A=-A_1$ a symmetric matrix,  $B=B_1$ is a positive definite matrix, $\Lambda =\{\x\in \mathbb{R}^d: \|\x\|=1\}$, and $\rho \geq 0$. Below, we derive the KL exponent of the merit function $\widehat{\Phi}_{GEP}$. To this end, we will use the following lemma from \cite{LP18}. Here we provide an alternative short proof for it.
\begin{lemma}
\label{l:bound}
Let $Q$ be a symmetric $d \times d$ matrix. Then there exists $c>0$ such that, for all $\x \in \mathbb{R}^d$, $\|Q\x\|^2 \geq c\, (\x^\top Q\x)$.
\end{lemma}
\begin{proof}
Let $Q=U^\top \Sigma U$ where $U$ is an orthonormal matrix and $\Sigma={\rm diag}(\lambda_1,\dots,\lambda_n)$ is a diagonal matrix whose diagonal elements are the eigenvalues of $Q$ with $\lambda_1 \leq \lambda_2 \leq \dots \leq \lambda_n$. Let $\x \in \mathbb{R}^d$, $\y:=U\x$, and 
$I_0=\{j:\lambda_j \neq 0\}$. Then $\x^\top Q\x= \sum_{j=1}^N \lambda_j y_j^2 =\sum_{j \in I_0} \lambda_j y_j^2$ and  
\begin{equation*}
\|Q\x\|^2=\x^\top (Q^\top Q)\x= (U\x)^\top \Sigma^2 (U\x)= \sum_{j=1}^N \lambda_j^2 y_j^2=\sum_{j \in I_0} \lambda_j^2 y_j^2.
\end{equation*}
Setting $c :=\min\{|\lambda_j|: j \in I_0\}$, we see that
\begin{equation*}
c\, (\x^\top Q\x)= \sum_{j \in I_0} c\lambda_j y_j^2 \leq \sum_{j \in I_0} c|\lambda_j| y_j^2 \leq \sum_{j \in I_0} |\lambda_j|^2 y_j^2=\|Q\x\|^2,
\end{equation*}
which completes the proof.
\end{proof}

Next we prove that the KL exponent of the merit function $\widehat{\Phi}_{GEP}$ is $\frac{1}{2}$. To do this, for an index set $J=\{j_1,\dots,j_k\} \subseteq \{1,\dots,d\}$ with $k \leq d$, we denote $\x_J:=(x_{j_1},\dots,x_{j_k})$. Moreover, for two index sets $I,J$, we denote
$A_{IJ}=(A_{ij})_{i \in I, j \in J}.$

\begin{theorem}
\label{t:KL_0norm_reg}
Consider the function
$\Phi(\x)=\frac{\x^\top A\x}{\x^\top B\x}+ \lambda \|\x\|_0 + \delta_{\Lambda}(\x)$,
where $\Lambda=\{\x: \|\x\|=1\}$, $A,B$ are symmetric matrices with $B$ positive definite, and $\lambda>0$. Then $\Phi$ is a KL function with exponent $\frac{1}{2}$. In particular, for all $\rho \geq 0$,
\begin{equation*}
\widehat{\Phi}_{GEP}(\x,\bu)= \frac{\x^\top A\x}{\x^\top B\x}+ \lambda \|\x\|_0 + \delta_{\Lambda}(\x)+ \rho \|\x-\bu\|^2
\end{equation*}
satisfies the KL property with exponent $\frac{1}{2}$ at $(\overline{\x},\overline{\x})$ for all $\overline{\x} \in \dom\partial_L \Phi$ .
\end{theorem}
\begin{proof}
Take any $\overline{\x} \in \dom\partial_L \Phi$. Then $\overline{\x}  \in \Lambda$.
Let $J=\supp(\overline{\x})$ and use $|J|$ to denote the cardinality of $J$.  Choose $\eta \in (0,1)$ such that, for all $\|\x-\overline{\x}\| < \eta$,
\begin{equation*}
\left|\, \frac{\overline{\x}^\top A\overline{\x}}{\overline{\x}^\top B\overline{\x}}-\frac{\x^\top A\x}{\x^\top B\x}\, \right| \leq \frac{\lambda}{4} \text{ and } \eta< \frac{\lambda}{4}.
\end{equation*} Let $\x$ with $\|\x-\overline{\x}\|<\eta$ and  $\Phi(\overline{\x})<\Phi(\x)<\Phi(\overline{\x})+\eta$.
We first see that, by shrinking $\eta$ if necessary, one can assume that
\begin{equation*}
J=\supp(\overline{\x})=\supp(\x).
\end{equation*}
Indeed, by continuity and by shrinking $\eta$ if necessary, one has $\supp(\overline{\x})\subseteq \supp(\x)$. Suppose that $\supp(\overline{\x}) \subsetneq \supp(\x)$. Then $\|\x\|_0>\|\overline{\x}\|_0$, and so,
$\|\x\|_0 \geq \|\overline{\x}\|_0+1$.
From our choice of $\x$, one has $\x \in \Lambda$ and
\begin{equation*}
\frac{\overline{\x}^\top A\overline{\x}}{\overline{\x}^\top B\overline{\x}}+ \lambda \|\overline{\x}\|_0< \frac{\x^\top A\x}{\x^\top B\x}+ \lambda \|\x\|_0 <\frac{\overline{\x}^\top A\overline{\x}}{\overline{\x}^\top B\overline{\x}}+ \lambda \|\overline{\x}\|_0 +\eta. \end{equation*}
This shows that
$\|\x\|_0<\frac{1}{2}+ \|\overline{\x}\|_0$,
which is impossible.

Using Lemma~\ref{l:subdiff} and noting that $J =\supp(\x)$, we derive that
\begin{multline*}
\partial_L \Phi(\x) \subseteq \Bigg\{\frac{2A\x(\x^\top B\x)-2B\x(\x^\top A\x)}{(\x^\top B\x)^2} +\lambda \bv +t \x:  t \in \mathbb{R}, \\
v_j=0 \text{~if~} j \in J, \text{~and~} v_j \in \mathbb{R} \text{~if~} j \notin J\Bigg\}.
\end{multline*}
Denoting $[a]_J=(a_j)_{j \in J} \in \mathbb{R}^{|J|}$, this implies that
\begin{equation*}
\dist(0, \partial_L \Phi(\x)) \geq \inf_{t \in \mathbb{R}} \left\|\left[\frac{2A\x(\x^\top B\x)-2B\x(\x^\top A\x)}{(\x^\top B\x)^2}\right]_J +t \x_J\right\|.
\end{equation*}
A direct verification shows that
\begin{equation*}
\x^\top  \left( \frac{2A\x(\x^\top B\x)-2B\x(\x^\top A\x)}{(\x^\top B\x)^2}\right) =0,
\end{equation*}
which, together with $J=\supp(\x)$, implies that
\begin{equation*}
(\x_J)^\top \left( \left[\frac{2A\x(\x^\top B\x)-2B\x(\x^\top A\x)}{(\x^\top B\x)^2}\right]_J\right) =0.
\end{equation*}
Therefore,
\begin{equation*}
\dist(0, \partial_L \Phi(\x)) \geq \left\|\left[\frac{2A\x(\x^\top B\x)-2B\x(\x^\top A\x)}{(\x^\top B\x)^2}\right]_J\right\|
=\frac{2}{\x^\top B\x} \left\| [A\x]_J -\frac{\x^\top A\x}{\x^\top B\x} [B\x]_J\right\|.
\end{equation*}
Using $J=\supp(\x)$ again, we have that
\begin{equation*}
[A\x]_J =A_{JJ}\x_J,\; \x^\top A\x =(\x_J)^\top A_{JJ}\x_J,\; [B\x]_J= B_{JJ}\x_J, \text{~and~} \x^\top B\x=(\x_J)^\top B_{JJ}\x_J,
\end{equation*}
and hence $\dist(0, \partial_L \Phi(\x))
\geq \frac{2}{\x^\top B\x} \left\| A_{JJ}\x_J -\frac{\x^\top A\x}{\x^\top B\x} 
B_{JJ}\x_J\right\|$.

Now, let $q(\z) =\z^\top A_{JJ}\z -\frac{\overline{\x}^\top A\overline{\x}}{\overline{\x}^\top B\overline{\x}}\, \z^\top B_{JJ}\z$ for $\z \in \mathbb{R}^{|J|}$.
Then
\begin{align*}
\dist(0, \partial_L \Phi(\x))
& \geq \frac{2}{\x^\top B\x} \left\| A_{JJ}\x_J -  \frac{\x^\top A\x}{\x^\top B\x} 
B_{JJ}\x_J\right\| \\
& \geq \frac{2}{\x^\top B\x} \left(\left\| A_{JJ}\x_J -  \frac{\overline{\x}^\top A\overline{\x}}{\overline{\x}^\top B\overline{\x}} 
B_{JJ}\x_J\right\| - \left|\frac{\x^\top A\x}{\x^\top B\x}-\frac{\overline{\x}^\top A\overline{\x}}{\overline{\x}^\top B\overline{\x}}
\right| \, \|B_{JJ}\x_J\|\right) \\
& = \frac{1}{\x^\top B\x} \|\nabla q (\x_J)\| -\frac{2}{\x^\top B\x}|\Phi(\x)-\Phi(\overline{\x})|\, \|B_{JJ}\x_J\|,
\end{align*}
where the second inequality follows from the triangle inequality and the last equality holds as $\x, \overline{\x} \in \Lambda$ and $J=\supp(\x)=\supp(\overline{\x})$ (and so, $\|\x\|_0=\|\overline{\x}\|_0$).

From Lemma \ref{l:bound}, there exists $c>0$ such that, for all $\z$,
$\|\nabla q(\z)\|^2 \geq c\, q(\z)$.
Indeed, one can set $c :=\min_{1 \leq j \leq |J|} \left \{4|\lambda_j(A_{JJ}-\frac{\overline{\x}^\top A\overline{\x}}{\overline{\x}^\top B\overline{\x}}B_{JJ})|: \lambda_j(A_{JJ}-\frac{\overline{\x}^\top A\overline{\x}}{\overline{\x}^\top B\overline{\x}}B_{JJ}) \neq 0 \right\} $, where $\lambda_j(Q)$ are the eigenvalues of a matrix $Q$. Noting that
\begin{align*}
\frac{q(\x_J)}{\x^\top B\x} &= \left[(\x_J)^\top A_{JJ}(\x_J)-\frac{\overline{\x}^\top A\overline{\x}}{\overline{\x}^\top B\overline{\x}}(\x_J)^\top B_{JJ}(\x_J)\right] \frac{1}{\x^\top B\x} \notag \\
& = \left[\x^\top A\x - \frac{\overline{\x}^\top A\overline{\x}}{\overline{\x}^\top B\overline{\x}} \x^\top B\x \right] \frac{1}{\x^\top B\x} \notag \\
& = \frac{\x^\top A\x}{\x^\top B\x}-\frac{\overline{\x}^\top A\overline{\x}}{\overline{\x}^\top B\overline{\x}} =\Phi(\x)-\Phi(\overline{\x})>0,
\end{align*}
one has
\begin{align*}
\dist(0, \partial_L \Phi(\x))
&\geq \frac{\sqrt{c} q(\x_J)^{1/2}}{\x^\top B\x} -2|\Phi(\x)-\Phi(\overline{\x})|\frac{\|B_{JJ}\x_J\|}{\x^\top B\x} \\
& = [\Phi(\x)-\Phi(\overline{\x})]^{1/2} \left(\frac{\sqrt{c}}{\sqrt{\x^\top B\x}} -2[\Phi(\x)-\Phi(\overline{\x})]^{1/2}\frac{\|B_{JJ}\x_J\|}{\x^\top B\x}\right).
\end{align*}
Let 
$c_1 :=\min\{\sqrt{\x^\top B\x}: \x \in \Lambda\}$ and $c_2 :=\max\{\sqrt{\x^\top B\x}: \x \in \Lambda\}$. By shrinking $\eta$ if necessary, we can assume that $\eta \in (0,1)$ and
\begin{equation*}
2[\Phi(\x)-\Phi(\overline{\x})]^{1/2}\frac{\|B_{JJ}\x_J\|}{\x^\top B\x}\leq 2\eta^{1/2}\frac{\|B_{JJ}\x_J\|}{c_1}\leq \frac{\sqrt{c}}{2c_2},
\end{equation*}
where the first inequality follows by the fact $\Phi(\overline{\x})<\Phi(\x)<\Phi(\overline{\x})+\eta$. Then, we see that
\begin{equation*}
\dist(0, \partial_L \Phi(\x))
\geq [\Phi(\x)-\Phi(\overline{\x})]^{1/2} \left(\frac{\sqrt{c}}{c_2} -\frac{\sqrt{c}}{2c_2}\right) =\frac{\sqrt{c}}{2c_2}[\Phi(\x)-\Phi(\overline{\x})]^{1/2}.
\end{equation*}
Thus, $\Phi$ satisfies the KL property with exponent $\frac{1}{2}$. This shows that, according to Lemma~\ref{l:LP},
$\widehat{\Phi}_{GEP}$ satisfies the KL property with exponent $\frac{1}{2}$ at $\overline{\x}$ for all $\overline{\x} \in \dom\partial_L \Phi$.
\end{proof}

\subsection{Generalized eigenvalue problem with sparsity constraint}

Consider the generalized eigenvalue problem with sparsity constraint
\begin{equation}
\max_{x\in \mathbb{R}^d} \frac{\x^\top A_1\x}{\x^\top B_1\x} 
\quad\text{s.t.}\quad \|\x\|=1, \, \|\x\|_0 \leq r,
\tag{GEPS}
\end{equation}
where $A_1,B_1$ are symmetric matrices such that $A_1$ is positive semidefinite and $B$ is positive definite, and $r>0.$ For this generalized eigenvalue problem with sparsity constraint, the corresponding merit function for the proposed Algorithm~\ref{algo:epasa} takes the form
\begin{equation*}
\widehat{\Phi}_{GEPS}(\x,\bu)=\frac{\x^\top A \x}{\x^\top B \x}+ \delta_{\Lambda \cap C_r}(\x)+ \rho \|\x-\bu\|^2,
\end{equation*}
where $A=-A_1$ is a symmetric matrix, $B=B_1$ is a positive definite matrix, $\Lambda =\{\x\in \mathbb{R}^d: \|\x\|=1\}$, $C_r =\{\x\in \mathbb{R}^d: \|\x\|_0 \leq r\}$ with $r>0$, and $\rho \geq 0$. Below, we investigate the KL exponent for this merit function.

\begin{theorem}\label{th:5.8}
Consider the function
$\Phi(\x) =\frac{\x^\top A\x}{\x^\top B\x}+ \delta_{\Lambda \cap C_r}(\x)$, 
where $\Lambda =\{\x\in \mathbb{R}^d: \|\x\|=1\}$, $C_r =\{\x\in \mathbb{R}^d: \|\x\|_0 \leq r\}$ and $A,B$ are symmetric matrices with $B$ positive definite. Then $\Phi$ is a KL function with exponent $\frac{1}{2}$. In particular, for all $\rho \geq 0$,
\begin{equation*}
\widehat{\Phi}_{GEPS}(\x,\bu) =\frac{\x^\top A\x}{\x^\top B\x} +\delta_{\Lambda \cap C_r}(\x) +\rho \|\x-\bu\|^2
\end{equation*}
satisfies the KL property with exponent $\frac{1}{2}$ at $(\overline{\x},\overline{\x})$ for all $\overline{\x} \in {\rm dom}\,  \partial_L \Phi$.
\end{theorem}
\begin{proof}
Take any $\overline{\x} \in \Lambda \cap C_r$. We split the proof into two cases: $\|\overline{\x}\|_0=r$ and $\|\overline{\x}\|_0<r$.

\emph{Case~1:} $\|\overline{\x}\|_0=r$. Let $\delta>0$ and take any $\x \in \Lambda \cap C_r$ with $\|\x-\overline{\x}\| \leq \delta$. By shrinking $\delta$ if necessary, we have
$\supp(\overline{\x}) \subseteq \supp({\x})$. So, $\|\x\|_0 \geq \|\overline{\x}\|_0=r$. As $x \in C_r$, we see that  $\|\x\|_0=\|\overline{\x}\|_0=r$ and so, $\supp(\overline{\x}) = \supp({\x})$. Then, a similar line of argument as in Theorem~\ref{t:KL_0norm_reg} gives the desired conclusion.

\emph{Case~2:} $\|\overline{\x}\|_0<r$. Let $\mathcal{I}=\{I \subseteq \{1,\dots,n\}: \supp(\overline{\x}) \subseteq I\}$. Clearly, $|\mathcal{I}|<+\infty$. Let $\delta>0$ and take any $\x \in \Lambda \cap C_r$ with $\|\x-\overline{\x}\| \leq \delta$. By shrinking $\delta$ if necessary, we have
$\supp(\overline{\x}) \subseteq \supp({\x})$, and so, $J_{\x}:=\supp({\x}) \in \mathcal{I}$. Let $\x$ with $\|\x-\overline{\x}\|<\eta$ and  $\Phi(\overline{\x})<\Phi(\x)<\Phi(\overline{\x})+\eta$.
From our choice of $\x$, one has $\x \in \Lambda$.
Moreover, using Lemma~\ref{l:subdiff}, a direct computation gives us that
\begin{multline*}
\partial_L \Phi(\x) \subseteq \Bigg\{\frac{2A\x(\x^\top B\x)-2B\x(\x^\top A\x)}{(\x^\top B\x)^2}+\lambda \bv + t \x :  t \in \mathbb{R}, \  \widehat{J} \subseteq \{1,\dots,n\}\smallsetminus J_{\x}, \\
|\widehat{J}|=r-|J_{\x}|,\, v_i=0 \text{ if } i \in J_{\x} \cup \widehat{J}, \, \text{and} \,   v_i \in \mathbb{R} \text{ if } i \notin \supp(\x) \cup \widehat{J} \Bigg\}.
\end{multline*}
It follows from $\x^\top \left(\frac{2A\x(\x^\top B\x)-2B\x(\x^\top A\x)}{(\x^\top B\x)^2}\right)=0$ that
\begin{equation*}
\x_{J_{\x} \cup \widehat{J}}^\top \left( \left[\frac{2A\x(\x^\top B\x)-2B\x(\x^\top A\x)}{(\x^\top B\x)^2}\right]_{J_{\x} \cup \widehat{J}}\right)=0.
\end{equation*}
Thus,
\begin{align*}
\dist(0, \partial_L \Phi(x))
&\geq \inf_{t \in \mathbb{R}, \widehat{J} \subseteq \{1,\dots,n\}\smallsetminus J_{\x},\, |\widehat{J}|=r-J_{\x}} \left\|\left[\frac{2A\x(\x^\top B\x)-2B\x(\x^\top A\x)}{(\x^\top B\x)^2}+ t \x \right]_{J_{\x} \cup \widehat{J}}\right\| \\
&= \inf_{\widehat{J} \subseteq \{1,\dots,n\}\smallsetminus J_{\x}, \, |\widehat{J}|=r-J_{\x}} \left\|\left[\frac{2A\x(\x^\top B\x)-2B\x(\x^\top A\x)}{(\x^\top B\x)^2}\right]_{J_{\x} \cup \widehat{J}}\right\| \\
&= \inf_{J \supseteq J_{\x}, |J|=r} \left\|\left[\frac{2A\x(\x^\top B\x)-2B\x(\x^\top A\x)}{(\x^\top B\x)^2}\right]_{J}\right\|.
\end{align*}

Using a similar line of argument as in Theorem~\ref{t:KL_0norm_reg}, one has
\begin{equation*}
\dist(0, \partial_L \Phi(x)) \geq \inf_{J \supseteq J_{\x}, |J|=r} \left\{\frac{1}{\x^\top B\x} \|\nabla q_J (\x_J)\| -\frac{2}{\x^\top B\x}|\Phi(\x)-\Phi(\overline{\x})| \, \|B_{JJ}\x_J\|\right\},
\end{equation*}
where $q_J(\z) =\z^\top A_{JJ}\z- \frac{\overline{\x}^\top A\overline{\x}}{\overline{\x}^\top B\overline{\x}}\z^\top B_{JJ}\z$ for $\z \in \mathbb{R}^{|J|}$.
By Lemma \ref{l:bound}, for each $J \supseteq J_{{\x}}$ with $|J|=r$, there exists $c_J>0$ such that $\|\nabla q_J(\z)\|^2 \geq c_J \, q_J(\z)$.
Note that $\{J: J \supseteq J_{{\x}} \text{ with } |J|=r\} \subseteq \mathcal{I}:=\{J: J \supseteq J_{\overline{\x}} \text{ with } |J|=r\}$ (as $\supp(\overline{\x})\subseteq \supp(\x)$) and $|\mathcal{I}|<+\infty$. So, $c:=\min_{J \in \mathcal{I}} c_J>0$. Noting from $\Phi(\x)-\Phi(\overline{\x})>0$, for each $J \supseteq J_{{\x}}$ with $|J|=r$, one has
\begin{align*}
\frac{q(\x_J)}{\x^\top B\x} &= \left[(\x_J)^\top A_{JJ}(\x_J)-\frac{\overline{\x}^\top A\overline{\x}}{\overline{\x}^\top B\overline{\x}}(\x_J)^\top B_{JJ}(\x_J)\right] \frac{1}{\x^\top B\x} \\
&= \left[\x^\top A\x - \frac{\overline{\x}^\top A\overline{\x}}{\overline{\x}^\top B\overline{\x}} \x^\top B\x \right] \frac{1}{\x^\top B\x} \\
&= \frac{\x^\top A\x}{\x^\top B\x}-\frac{\overline{\x}^\top A\overline{\x}}{\overline{\x}^\top B\overline{\x}}= \Phi(\x)-\Phi(\overline{\x})>0.
\end{align*}
Therefore, 
\begin{align*}
\dist(0, \partial_L \Phi(x))
&\geq \inf_{J \supseteq J_{\x}, |J|=r} \left\{\frac{\sqrt{c}q_J(x_J)^{1/2}}{\x^\top B\x} -2|\Phi(\x)-\Phi(\overline{\x})| \frac{\|B_{JJ}\x_J\|}{\x^\top B\x} \right\} \\
&= [\Phi(\x)-\Phi(\overline{\x})]^{1/2}\!\!\! \inf_{J \supseteq J_{\x}, |J|=r}\!\! \left\{\!\!\frac{c}{\sqrt{\x^\top B\x}} -2[\Phi(\x)-\Phi(\overline{\x})]^{1/2} \frac{\|B_{JJ}\x_J\|}{\x^\top B\x}\! \right\}.
\end{align*}
Following a similar line or arguments as in Theorem~\ref{t:KL_0norm_reg}, we get the desired conclusion.
\end{proof}

\begin{remark}[linear convergence of Algorithm~\ref{algo:epasa}]
\label{remark:linear_convergence}
In view of Theorem~\ref{t:global}, Theorem~\ref{t:KL_0norm_reg}, and Theorem~\ref{th:5.8}, we see that Algorithm~\ref{algo:epasa} exhibits linear convergence when applied to generalized eigenvalue problems with cardinality regularization and generalized eigenvalue problems with sparsity constraints.
\end{remark}

\section{Numerical examples}
\label{s:numerical}
In the section, we illustrate our proposed method via numerical examples. We first start with an explicit analytic example and use it to demonstrate the behavior of Algorithm~\ref{algo:epasa}  as well as the effect of the inertial parameters. Then, we examine the performance of the algorithm for the sparse eigenvalue optimization model. All the numerical tests are conducted on a computer with a 2.8 GHz Intel Core i7 and 8 GB RAM, equipped with MATLAB R2015a.

\subsection{Analytic examples}

Consider the problem
\begin{equation}\label{eq:ex}
\max_{\x \in \mathbb{R}^m} \left(m+1-\sum_{i=1}^m x_i\right)\prod_{i=1}^m x_i +\gamma \, \sum_{i=1}^m \frac{x_i+1}{x_i^2+2x_i+5} 
\quad\text{s.t.}\quad 0 \leq \x \leq 10,
\tag{EP}
\end{equation}
where $\gamma>0$. We first note that, for all $i \in \{1,\dots,m\}$, $x_i^2+2x_i+5=(x_i+1)^2+2^2 \geq 4(x_i+1)$ and that if $m+1-\sum_{i=1}^m x_i <0$, then $(m+1-\sum_{i=1}^m x_i)\prod_{i=1}^m x_i \leq 0$; otherwise, applying the Arithmetic Mean Geometric Mean (AM-GM) inequality to $(m+1)$ numbers $(m+1-\sum_{i=1}^m x_i)$, $x_1$, \dots, $x_m$ yields $(m+1-\sum_{i=1}^m x_i)\prod_{i=1}^m x_i \leq 1$. Direct verification shows that $\overline{\x}=(1,\dots,1)$ is the global solution of this problem. This example satisfies Assumption~\ref{a:standing} with $f_i(x_i) =\gamma(x_i+1)$, $g_i(x_i) =x_i^2+2x_i+5$, $\alpha_i=\frac{1}{4}$, and $\beta_i=2$ for all $i \in \{1,\dots,m\}$. Let $\gamma=10$, $\x_0=\x_{-1}$, $\delta=1$, and, for all $n \in \mathbb{N}$, $\nu_n=0$, $\tau_n =\delta+ \max_{1 \leq i \leq m}\{y_{i,n}\alpha_i+\frac{1}{2}y_{i,n}^2 \beta_i\} \leq \overline{\tau} :=\delta+ \max_{1 \leq i \leq m}\{\alpha_i \frac{\sqrt{M_i}}{m_i}+\frac{1}{2} \beta_i \frac{M_i}{m_i^2}\}$, where $M_i =\max_{0 \leq x_i \leq 10} f_i(x_i) =110$ and $m_i =\min_{0 \leq x_i \leq 10} g_i(x_i) =5$, $i \in \{1,\dots,m\}$. Then, for all $n \in \mathbb{N}$ and all $i \in \{1,\dots,m\}$, $z_{i,n} =x_{i,n} +\nu_n (x_{i,n}-x_{i,n-1}) =x_{i,n}$ and $w_{i,n} =\frac{\gamma(-x_{i,n}^2-2x_{i,n}+3)}{(x_{i,n}^2+2x_{i,n}+5)^2}$, and
\begin{align*}
x_{i,n+1} &=\argmax_{0\leq x_i\leq 10} \left\{ x_i\left(m+1 -x_i -s_{i,n}\right)p_{i,n} -\tau_n\left(x_i -z_{i,n}-\frac{1}{2 \tau_n}w_{i,n}\right)^2\right\} \\ 
&={\rm P}_{[0,10]}\left( \frac{2\tau_n\left(z_{i,n}+\frac{1}{2 \tau_n}w_{i,n}\right) +(m+1 -s_{i,n})p_{i,n}}{2\tau_n +2p_{i,n}} \right),
\end{align*}
where $s_{i,n} :=\sum_{j=1}^{i-1} x_{j,n+1} +\sum_{j=i+1}^m x_{j,n}$, $p_{i,n} :=\prod_{j=1}^{i-1} x_{j,n+1} \prod_{j=i+1}^m x_{j,n}$, and ${\rm P}_C$ denotes the Euclidean projection onto $C$. 

We randomly generate initial points in $[0,10]^m$ and perform Algorithm~\ref{algo:epasa}. For all the initial points, the algorithm produces a sequence $(\x_n)_{n\in \mathbb{N}}$ converging to the global maximizer. 
Figure~\ref{fig:Distance_Example1} depicts the convergence behavior for the case $m=2$ and $\gamma=10$, with  initial points $(0,0)$, $(0,1)$, $(1,0)$, and $(10,10)$ by plotting out the Euclidean distance to the solution $(1,1)$ per iteration.

\begin{figure}[htb!]
\begin{center}
 \includegraphics[width=0.78\columnwidth]{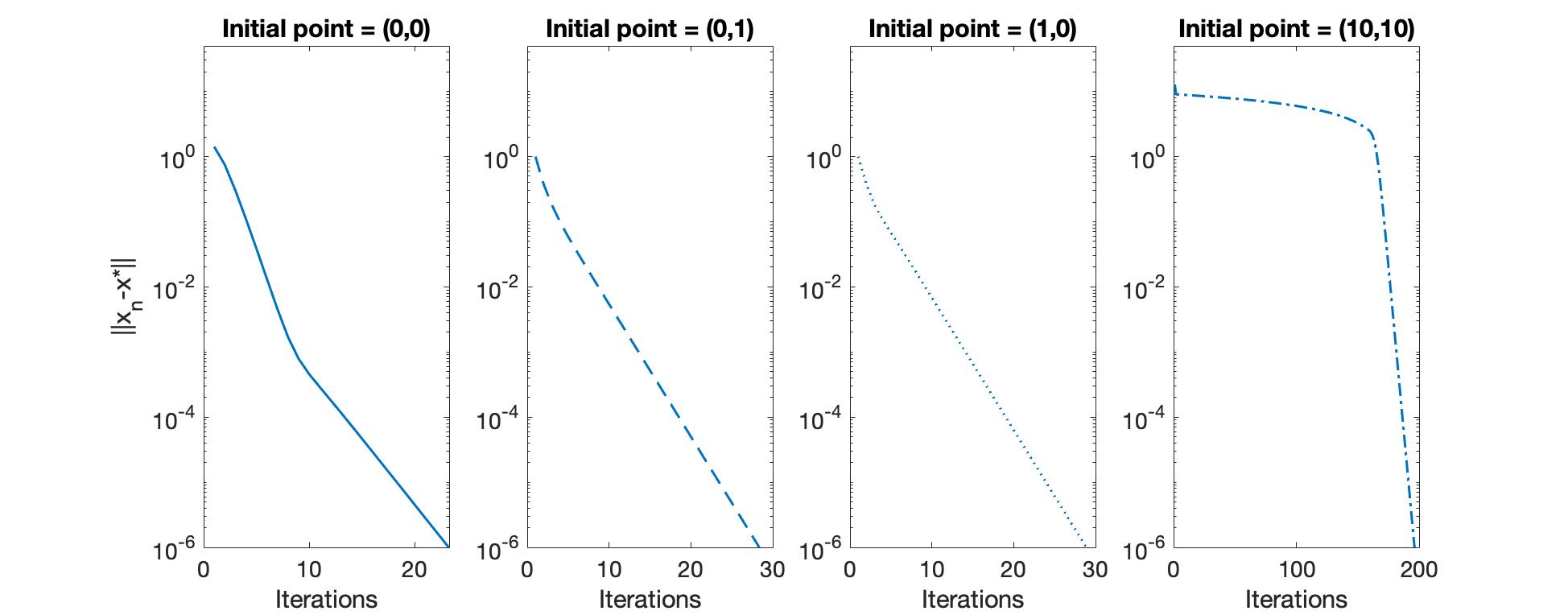}
 \end{center}
 \caption{Euclidean distance between the solution and the sequence generated by Algorithm~\ref{algo:epasa} for \eqref{eq:ex}}
 \label{fig:Distance_Example1}
\end{figure}

{\bf Effect of the inertial parameters.} We now illustrate the behavior of Algorithm~\ref{algo:epasa} by varying
the inertial parameters. To do this, we fix $m=2$ and $\gamma=10$ and an $\alpha \in (0,1)$. We set $\nu_n =\alpha \frac{\delta}{2\tau_n} <\frac{\delta}{2\tau_n}$. Starting with the initialization $\x_0=(10,10)$, we then run
Algorithm~\ref{algo:epasa} with different values for $\alpha\in [0, 1)$. Figure~\ref{fig:inertia_Example1} depicts the distance, in the log scale, between the sequence of iterates $(\x_n)_{n\in \mathbb{N}}$ and the solution $\overline{\x}=(1,1)$, for $\alpha \in \{0, 0.3, 0.6, 0.9\}$. As one can see from the figure, as $\alpha$ increases and approaches $1$, the algorithm
tends to converge faster.
\begin{figure}[htb!]
\begin{center}
 \includegraphics[width=0.48\columnwidth]{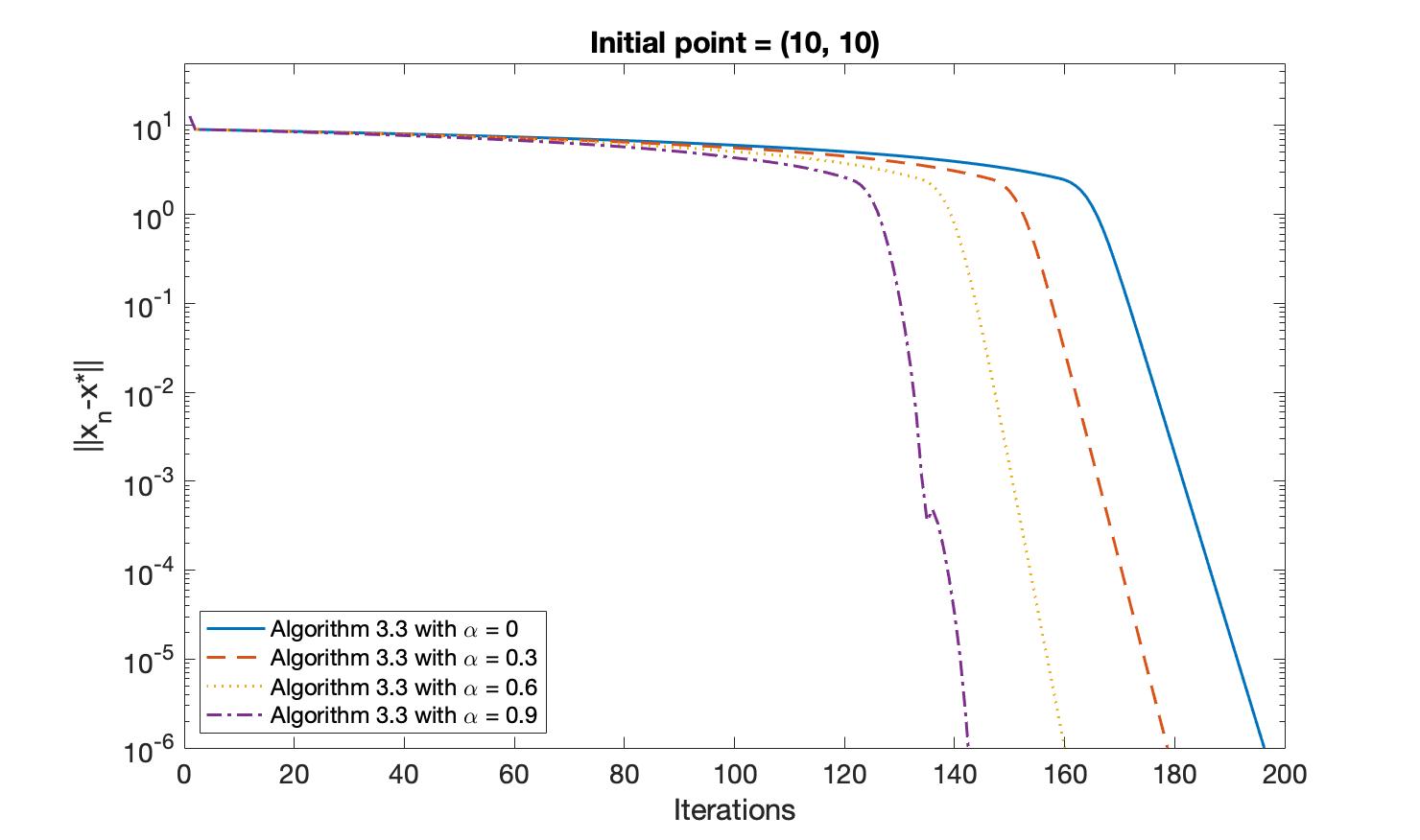}
 \end{center}
\caption{Illustration for different inertial parameters in solving \eqref{eq:ex} via Algorithm~\ref{algo:epasa}}
\label{fig:inertia_Example1}
\end{figure}

\subsection{Sparse generalized eigenvalue problems}

As another illustration of our algorithm, following~\cite{TWLZ18}, we consider a sparse generalized eigenvalue problem that arises from binary classification using sparse Fisher discriminant analysis. Consider $p$ observations $\z_1,\dots,\z_p$ with $\z_i \in \mathbb{R}^d$, $i\in \{1,\dots,p\}$, each of which belongs to one of two distinct classes. Let $I_k \subseteq \{1, \dots, p\}$ contain the indices of the
observations in class $k$, with $p_k =|I_k|$, $k=1,2,$ and $p_1+p_2=p$.
Let $\widehat{\bmu}_k =\frac{1}{p_k}\sum_{i \in I_k} \z_i$, for $k=1,2$. The so-called within-class and between-class covariance matrices are given by
\begin{equation*}
V_w =\frac{1}{p} \sum_{k=1}^2 \sum_{i \in I_k}(\z_i-\widehat{\bmu}_k)(\z_i-\widehat{\bmu}_k)^\top \text{ and } V_b =\frac{1}{p}\sum_{k=1}^2 p_k \widehat{\bmu}_k \widehat{\bmu}_k^\top.
\end{equation*}
The classification problem using sparse Fisher discriminant analysis (SFDA) then  seeks a low dimensional projection of the observations such that the between-class variance is large relative to the within-class variance. Mathematically, it solves
\begin{equation}\label{eq:SFDA}
\max_{\x \in \mathbb{R}^d} \frac{\x^\top V_b\x}{\x^\top V_w\x} -\lambda\phi(\x)
\quad\text{s.t.}\quad \|\x\|=1,
\tag{SFDA}
\end{equation}
where $\phi$ is a regularization function inducing sparsity, and $\lambda>0$.
This is a sparse generalized eigenvalue problems with $A=V_b$ and $B=V_w$. Here, we consider two specific sparse regularization functions: $\phi(\x)=\|\x\|_0$, and $\phi(\x)=\delta_{C_r}(\x)$ with $C_r =\{\x\in \mathbb{R}^d: \|\x\|_0 \leq r\}$ and $r>0$.

In the case where $\phi(\x)=\delta_{C_r}(\x)$, \cite{TWLZ18} proposed a truncated Rayleigh flow method (TRFM) for solving the above sparse generalized eigenvalue problem and showed the linear convergence of this method when the initial point $\x_0$ is close enough to a global solution. We note that, in general, it is hard to theoretically guarantee whether an initial point $\x_0$ is chosen to be close enough to a global solution, in order to ensure the convergence of the algorithm. On the other hand, Algorithm~\ref{algo:epasa} can be applied to \eqref{eq:SFDA} with both $\phi(\x)=\|\x\|_0$ and $\phi(\x)=\delta_{C_r}(\x)$, and Remark~\ref{remark:linear_convergence} shows that Algorithm~\ref{algo:epasa} converges linearly regardless of the choice of the initial points.

\subsubsection{Sparsity constrained case}
In this subsection, we consider the generalized eigenvalue problem with sparsity constraints, that is, \eqref{eq:SFDA} with $\phi(\x)=\delta_{C_r}(\x)$.
In this setting, Algorithm~\ref{algo:epasa} reads as
\begin{multline*}
\x_{n+1}\in {\rm P}_{\Lambda \cap C_r}\left(\z_n+ \frac{1}{2\tau_n} \frac{\x_n^\top V_b\x_n}{(\x_n^\top V_w\x_n)^2} \left[\frac{\x_n^\top V_w\x_n}{\x_n^\top V_b\x_n}V_b\x_n-V_w\x_n\right] \right) \\ \text{ with } \z_n=\x_n+\nu_n (\x_n-\x_{n-1}).
\end{multline*}
It is known that, for all $\ba =(a_1,\dots,a_d) \in \mathbb{R}^d$, $({\rm P}_{C_r}(\ba))_i =a_i$ for the $r$ largest components in absolute value of $\ba$, and $({\rm P}_{C_r}(\ba))_i =0$ otherwise. Then
\begin{equation*}
{\rm P}_{\Lambda \cap C_r}(\ba) =\begin{cases}
\left\{ \frac{\bv}{\|\bv\|}: \bv \in {\rm P}_{C_r}(\ba)\right\} & \text{if~} \ba \neq 0, \\
\Lambda \cap C_r & \text{if~} \ba=0.
\end{cases}
\end{equation*}
This can be seen, for example, by noting that ${\rm P}_{\Lambda \cap C_r}(\ba) =\argmin\{\frac{1}{2}\|\x-\ba\|^2: \x \in \Lambda \cap C_r\}=\argmin\{\langle \ba, \x\rangle: \x \in \Lambda \cap C_r\}$, and applying \cite[Proposition~13]{LT13}.

In our simulation, we adopt the same setting as in \cite{TWLZ18}: we set $\bmu_1 = {\bf 0}$, $\bmu_2=(\mu_{2,1},\dots,\mu_{2,d})^\top$ with  $\mu_{2,j} = 0.5$ for $j \in \{2,4,\dots,40\}$ and $\mu_{2,j} = 0$ otherwise. Let $\Sigma$ be a block diagonal covariance matrix with five blocks, each
of dimension $(d/5 \times d/5)$. The $(j,j')$-th element of each block takes the value $0.8^{|j-j'|}$. As explained in \cite{TWLZ18}, this covariance structure is intended to mimic the covariance structure of gene expression data. The observation data are simulated as $\z_i\sim N(\mu_k, \Sigma)$ for
$i \in I_k$, $k=1,2$.

We use our proposed inertial proximal subgradient method (Algorithm~\ref{algo:epasa}) and the truncated Rayleigh flow method (TRFM) for solving \eqref{eq:SFDA} with $\phi(\x)=\delta_{C_r}(\x)$, where we set $r=50$,   $p_1=p_2=500$, $p=p_1+p_2=1000$, and $d=2000$.

\begin{itemize}
    \item For Algorithm~\ref{algo:epasa}, we use the  initial point
 $\x_0=(\underbrace{{1}/{\sqrt{r}},\dots,{1}/{\sqrt{r}}}_r,0,\dots,0) \in \mathbb{R}^d$. Direct verification shows that Assumption~\ref{a:standing} is satisfied with $\alpha_1=0$ and $\beta_1=2\lambda_{\max}(V_w)$. So, by Remark \ref{remark:parameter},  we can set $\delta=1$, $\tau_n=1+\frac{\x_n^\top V_b\x_n}{(\x_n^\top V_w\x_n)^2} \lambda_{\max}(V_w)$,  $\overline{\nu}=0.4999<\frac{\delta}{2}$ and $\nu_n=\frac{\overline{\nu}}{\tau_n}$. We stop the algorithm when either the iterations reach the maximum iteration number $6000$ or the quantity $\|\x_{n+1}-\x_n\|$ is less than
$10^{-6}$.
\item For (TRFM), we use the same initial point $\x_0$ as in Algorithm~\ref{algo:epasa}. We  also use the same termination criteria as in Algorithm~\ref{algo:epasa}.
\end{itemize}

We run TRFM and Algorithm~\ref{algo:epasa} for 50 trials. Table~\ref{tab:computation} summarizes the output of the two methods by listing the average value for
\begin{enumerate}
\item
the objective value of the computed solution;
\item
the CPU time measured in seconds;
\item
the number of iterations used (round to the nearest integer).
\end{enumerate}

\begin{table}[!htb]
\centering
\caption{Computation results for \eqref{eq:SFDA} with sparsity constraint}
\label{tab:computation}
\begin{tabular}{l c c c} \hline
& Objective value of & & Number of \\
& computed solution & CPU time & iterations \\ \hline
TRFM & 12.2932 & 6.9976 & 1083  \\ 
Algorithm~\ref{algo:epasa} & 12.5461 & 4.8148 & 555 \\ \hline
\end{tabular}

\end{table}
From Table~\ref{tab:computation}, one can see that Algorithm~\ref{algo:epasa} is competitive with the TRFM method and produces a solution with better quality in terms
of the final objective value (note that \eqref{eq:SFDA} is a maximization problem). Moreover, Algorithm~\ref{algo:epasa} also uses less CPU time and number of iterations. As an illustration, we also plot  $\|\x_n-\x^*\|$ against the number of iterations $n$, in logarithmic scale,
where $\x^*$ is the  approximated solution produced by the corresponding algorithm. Figure~\ref{fig:vsTRFM} supports the theoretical finding that Algorithm~\ref{algo:epasa} exhibits linear convergence in this case.

\begin{figure}[htb!]
\begin{center}
\includegraphics[width=0.48\columnwidth]{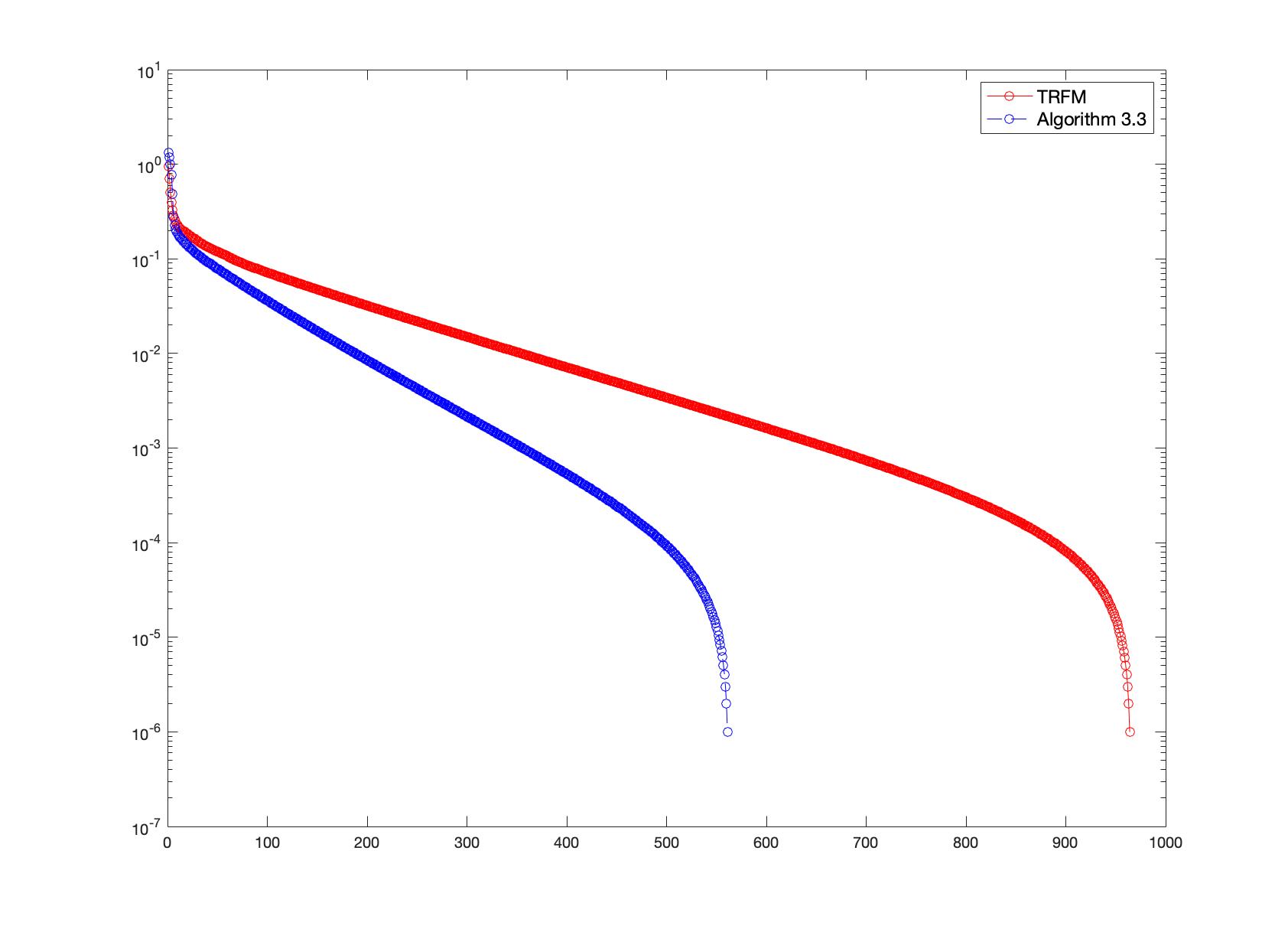}
\end{center}
\caption{Euclidean distance between $\x_n$ and $\x^*$ in every iteration}
\label{fig:vsTRFM}
\end{figure}

\subsubsection{Sparse generalized eigenvalue problem with cardinality regularization}
In this subsection, we consider the generalized eigenvalue problem with cardinality regularization, that is, \eqref{eq:SFDA} with $\phi(\x)= \|\x\|_0$.
In this setting, Algorithm~\ref{algo:epasa} reads as
\begin{align*}
\x_{n+1} &= \argmax_{\|\x\|=1} \left\{-\lambda \|\x\|_0 -\tau_n \left\|\x-\z_{n}-\frac{1}{2\tau_n} \bw_n \right\|^2 \right\} \\
&= \argmax_{\|\x\|=1} \left\{-\lambda \|\x\|_0 + \langle 2\tau_n\z_{n}+ \bw_n, \x\rangle \right\}
\end{align*}
with $\lambda>0$, $\z_n=\x_n+\nu_n (\x_n-\x_{n-1})$, and
\begin{equation*}
\bw_n =\frac{\x_n^\top V_b\x_n}{(\x_n^\top V_w\x_n)^2} \left[\frac{\x_n^\top V_w\x_n}{\x_n^\top V_b\x_n}V_b\x_n-V_w\x_n\right].
\end{equation*}
We note that, for each $\ba \in \mathbb{R}^d$, the optimization problem $\argmax_{\|\x\|=1} \{-\lambda \|\x\|_0 +\langle \ba, \x \rangle\}$ has a closed form solution \cite[Proposition~6]{SBP15}. In our numerical experiment, we set $\lambda =0.035$. We also  generate the data as in the previous subsection, using    the same initial point, parameters $\tau_n$, $\nu_n$ and $\delta$, and  employing the same termination criteria.

We run  Algorithm~\ref{algo:epasa} for 50 trials. Table~\ref{tab:computation1}  summarizes the output of the method where the meanings of the items are the same as in the previous subsection.
\begin{table}[!htb]
\centering
\caption{Computation results for \eqref{eq:SFDA} with cardinality regularization}\label{tab:computation1}
\begin{tabular}{c c c} \hline
Objective value of & & Number of \\
computed solution & CPU time & iterations\\ \hline
13.7196 & 3.1013 & 1074 \\ \hline
\end{tabular}
\end{table}

We also plot out Euclidean distance between $\x_n$ and $\x^*$ per iteration in log scale (Figure~\ref{fig:cadinality}), which supports the theoretical finding that Algorithm~\ref{algo:epasa} exhibits linear convergence for this problem.
\begin{figure}[htb!]

\begin{center}
\includegraphics[width=0.48\columnwidth]{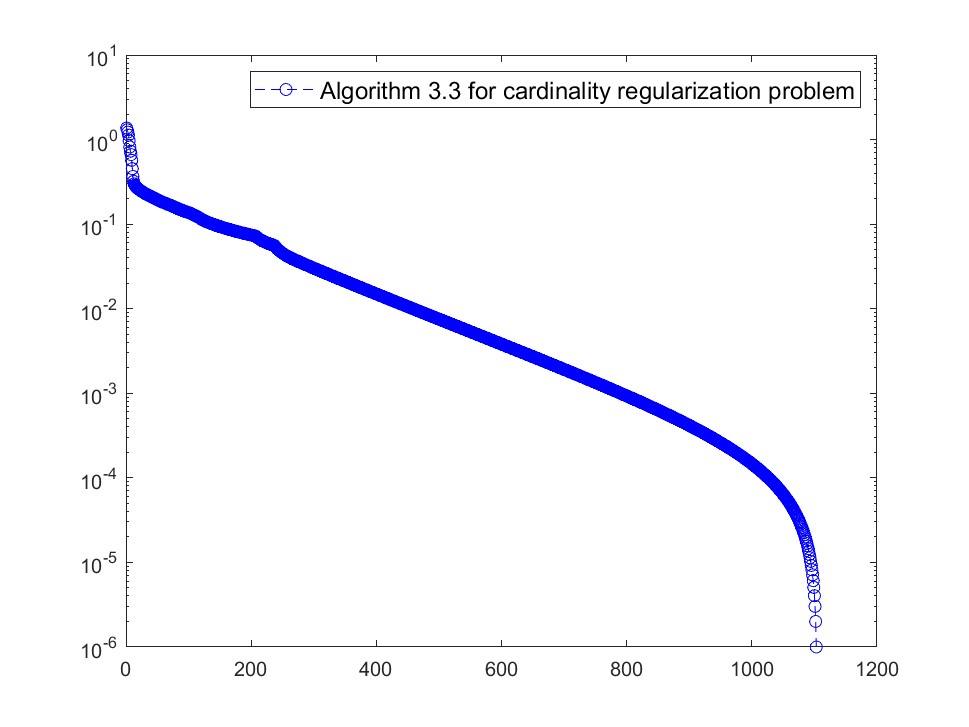}
\end{center}
\caption{Euclidean distance between $\x_n$ and $\x^*$ in every iteration}
\label{fig:cadinality}
\end{figure}

\appendix
\section{Proof of Lemmas~\ref{l:calrules}, \ref{l:subdiff}, \ref{l:stationary}, and \ref{l:equi}}
\label{s:appendix}

\begin{proof}[Proof of Lemma~\ref{l:calrules}]
\ref{l:calrules_separable}: This is given in \cite[Proposition 10.5]{RW98}.

\ref{l:calrules_sum}: This follows from \cite[Corollary~10.9]{RW98}.

\ref{l:calrules_sign}: This is an application of \cite[Corollary~3.4]{MNY06} with $\varphi_1 \equiv 0$ and $\varphi_2 =f$. 

\ref{l:calrules_quotient}: We first have from \cite[Proposition~1.111(ii)]{Mor06} and \ref{l:calrules_sum} that 
\begin{equation}\label{eq:quotient}
\partial_L\left(\frac{-f}{g}\right)(\x) =\frac{\partial_L(-g(\x)f +f(\x)g)(\x)}{g(\x)^2} \subseteq \frac{\partial_L(-g(\x)f)(\x) +\partial_L(f(\x)g)(\x)}{g(\x)^2}.
\end{equation}
Assume that $\widehat{\partial} f$ is nonempty-valued around $\x$. Then, if $g(\x) >0$, $\partial_L(-g(\x)f)(\x) =g(\x)\partial_L(-f)(\x) \subseteq -g(\x)\partial_L f(\x)$ due to \ref{l:calrules_sign}. If $g(\x) <0$, then $-g(\x) >0$ and $\partial_L(-g(\x)f)(\x) =-g(\x)\partial_L f(\x)$. Thus, we obtain the desired inclusion.

Now, assume that $f$ is strictly differentiable at $\x$. Then, by combining \eqref{eq:quotient} with the last assertion in \ref{l:calrules_sum}, $\partial_L\left(\frac{-f}{g}\right)(\x) =\frac{-g(\x)\nabla f(\x) +\partial_L(f(\x)g)(\x)}{g(\x)^2}$.
On the other hand, we have from \cite[Corollaries 1.12.2 and 1.14.2]{Kru03} that $\widehat{\partial}\left(\frac{-f}{g}\right)(\x) =\frac{-g(\x)\nabla f(\x) +\widehat{\partial}(f(\x)g)(\x)}{g(\x)^2}$.
The remaining conclusion follows from these two equalities.

\ref{l:calrules_sqrt}: The chain rule is given in \cite[Theorem~1.110(ii)]{Mor06}. The two square root rules follow by letting $\theta(t) =\sqrt{t}$ and $\theta(t) =-\sqrt{t}$, respectively.
\end{proof}

\begin{proof}[Proof of Lemma~\ref{l:subdiff}]
\ref{l:subdiff_l0}: The formula for Fr\'echet and limiting subdifferentials of $\|\cdot\|_0$ can be found in \cite[Section~3]{Le13}. The formula for the horizon subdifferential can be verified directly.

\ref{l:subdiff_sphere}: This follows by a direct verification.

\ref{l:subdiff_ball}: The limiting subdifferential formula for $\delta_{C_r}$ can be found in \cite[Theorem~3.9]{BLPW14}. The formula for the horizon subdifferential can be verified directly.

\ref{l:subdiff_sum}\&\ref{l:subdiff_sum'}: We deduce from \ref{l:subdiff_l0}, \ref{l:subdiff_sphere}, and \ref{l:subdiff_ball} that, for all $\x\in \Lambda$, $\|\cdot\|_0$ and $\delta_{\Lambda}$ are regular at $\x$ and $\left(-\partial_L^{\infty} (\|\cdot\|_0)(\x)\right) \cap \partial_L^{\infty} \delta_{\Lambda}(\x) =\{0\}$, and that, for all $\x\in \Lambda \cap C_r$, $\left(-\partial_L^{\infty} \delta_{C_r}(\x)\right) \cap \partial_L^{\infty} \delta_{\Lambda}(\x) =\{0\}$. The conclusions then follow from Lemma~\ref{l:calrules}\ref{l:calrules_sum}.
\end{proof}

\begin{proof}[Proof of Lemma~\ref{l:stationary}]
Let us first consider the case when $h$ is strictly differentiable at $\overline{\x}$. By Lemma~\ref{l:calrules}\ref{l:calrules_sum}, $\partial_L (-h+\delta_S)(\overline{x}) =-\nabla h(\overline{\x}) +\partial_L \delta_S(\overline{\x})$. Since $\delta_S(\x) =\delta_{S_1}(x_1) +\dots +\delta_{S_m}(x_m)$, we learn from Lemma~\ref{l:calrules}\ref{l:calrules_separable} that $\partial_L \delta_S(\overline{\x}) =\partial_L^{x_1} \delta_S(\overline{x})\times \dots \times \partial_L^{x_m} \delta_S(\overline{x})$, and so 
\begin{equation*}
\partial_L (-h+\delta_S)(\overline{x}) =\partial_L^{x_1}(-h+\delta_S)(\overline{x})\times \dots \times \partial_L^{x_m}(-h+\delta_S)(\overline{x}).    
\end{equation*}
This equality is obvious in the case when $m =1$. 
Next, since $F(\x) =h(\x) +\sum_{i=1}^m \frac{f_i(x_i)}{g_i(x_i)}$ with each $\frac{f_i}{g_i}$ Lipschitz continuous around $\overline{x}_i$, again using Lemma~\ref{l:calrules}\ref{l:calrules_separable}\&\ref{l:calrules_sum}, we have that 
\begin{align}\label{eq:009}
& \ \partial_L (-F+\delta_S)(\overline{\x}) \subseteq \partial_L (-h+\delta_S)(\overline{\x}) +\partial_L \left(-\sum_{i=1}^m \frac{f_i}{g_i}\right)(\overline{\x}) \\
=& \ \partial_L^{x_1}(-h+\delta_S)(\overline{\x})\times \dots \times \partial_L^{x_m}(-h+\delta_S)(\overline{\x}) +\partial_L\left(\frac{-f_1}{g_1}\right)(\overline{x}_1)\times \dots \times \partial_L\left(\frac{-f_m}{g_m}\right)(\overline{x}_1). \notag
\end{align}

\ref{l:stationary_imply}: Assume that, for each $i \in \{1,\dots,m\}$, $\widehat{\partial} f_i$ is nonempty-valued around $\overline{x}_i$. Then, by Lemma~\ref{l:calrules}\ref{l:calrules_quotient}, for each $i \in \{1,\dots,m\}$,
\begin{equation}\label{eq:00}
\partial_L \left(\frac{-f_i}{g_i}\right)(\overline{x}_i) \subseteq \frac{-g_i(\overline{x}_i)\partial_L f_i(\overline{x}_i) +f_i(\overline{x}_i)\partial_L g_i(\overline{x}_i)}{g_i(\overline{x}_i)^2}.
\end{equation}
In view of \eqref{eq:009} and \eqref{eq:00}, if $\overline{\x}$ is a stationary point for \eqref{eq:prob}, then it is a lifted coordinate-wise stationary point for \eqref{eq:prob}.

\ref{l:stationary_equi}: By Lemma~\ref{l:calrules}\ref{l:calrules_separable},\ref{l:calrules_sum}\&\ref{l:calrules_quotient}, the inclusions in \eqref{eq:009} and \eqref{eq:00} can be replaced by equalities. The conclusion then follows.
\end{proof}

\begin{proof}[Proof of Lemma~\ref{l:equi}]
For each $i\in \{1,\dots,m\}$, set $H_i(x_i,y_i) :=2y_i\sqrt{f_i(x_i)} -y_i^2g_i(x_i)$.

\ref{l:equi_global}: This follows from the observation that
\begin{equation*}
\max_{\y\in \mathbb{R}^m} H(\x,\y) =\sum_{i=1}^m \max_{y_i\in \mathbb{R}} H_i(x_i,y_i) =\sum_{i=1}^m H_i\left(x_i,\frac{\sqrt{f_i(x_i)}}{g_i(x_i)}\right) =\sum_{i=1}^m\frac{f_i(x_i)}{g_i(x_i)}.   
\end{equation*}

\ref{l:equi_local}: Assume that, for each $i \in \{1,\dots,m\}$, $\widehat{\partial} f_i$ is nonempty-valued around $\overline{x}_i$. Then, since $f_i(\overline{x}_i) >0$ and $\overline{y}_i\geq 0$, we have from Lemma~\ref{l:calrules}\ref{l:calrules_sum}, Lemma~\ref{l:calrules}\ref{l:calrules_sqrt}, and then Lemma~\ref{l:calrules}\ref{l:calrules_sign} that
\begin{align}\label{eq:use}
\partial_L^{x_i} (-H)(\overline{\x},\overline{\y}) &=\partial_L^{x_i} (-H_i)(\overline{x}_i,\overline{y}_i) \subseteq \frac{\overline{y}_i\, \partial_L (-f_i)(\overline{x}_i)}{\sqrt{f_i(\overline{x}_i)}}+\overline{y}_i^2 \, \partial_L g_i(\overline{x}_i) \notag \\
&\subseteq \frac{-\overline{y}_i\, \partial_L f_i(\overline{x}_i)}{\sqrt{f_i(\overline{x}_i)}}+\overline{y}_i^2 \, \partial_L g_i(\overline{x}_i) 
= \frac{-g_i (\overline{x})\, \partial_L f_i(\overline{x}_i)+f_i(\overline{x}_i) \partial_L g_i(\overline{x}_i) }{g_i(\overline{x}_i)^2}.
\end{align}
As a result, if $(\overline{\x},\overline{\y})$ is a lifted coordinate-wise stationary point for \eqref{eq:prob1}, then $\overline{\x}$ is a lifted coordinate-wise stationary point for $\eqref{eq:prob}$.

Now, assume that, for each $i \in \{1,\dots,m\}$, $f_i$ is strictly differentiable at $\overline{x}_i$. Then the inclusions in \eqref{eq:use} become equalities, and the conclusion follows.
\end{proof}

\section*{Acknowledgment} 
The authors would like to thank the anonymous referees for various constructive comments and suggestions that helped improve the manuscript. The authors are also grateful for Dr. Qia Li for the help and discussions for the applications of sparse Fisher discriminant analysis.

\end{document}